%% file: anchored.tex
\documentclass[11pt]{article}
\usepackage{geometry}                
\usepackage{fullpage}
\usepackage{fontawesome}
\usepackage{graphicx}
\usepackage{amssymb}
\usepackage{epstopdf}
\usepackage{setspace}
\usepackage[abs]{overpic}
\usepackage{mathtools}
\usepackage[dvipsnames]{xcolor}

\usepackage{amsmath}
\usepackage{amsmath,amssymb,tikz-cd}
\usepackage{color,soul}
\usepackage{mathrsfs}
\usepackage{amsthm}
\usepackage[T1]{fontenc}
\usepackage[all,cmtip]{xy}
\usepackage[normalem]{ulem}
\usepackage{stmaryrd}
\usepackage{enumitem}
\usepackage{color}
\usepackage{hyperref}
\usepackage[title]{appendix}
\usepackage{tabularx}
\usepackage{amscd}
\usepackage{subcaption}
\usepackage{comment}
\usepackage{xcolor}
\usepackage{cancel}
\usepackage{mathabx} 
\newcommand{\C}{\mathbb{C}}
\newcommand{\Q}{\mathbb{Q}}
\newcommand{\R}{\mathbb{R}}
\newcommand{\Z}{\mathbb{Z}}
\newcommand{\A}{\mathcal{A}}

\newcommand{\Ker}{\operatorname{Ker}}
\newcommand{\op}{\operatorname}

\theoremstyle{plain}
\newtheorem{thm}{Theorem}

\newtheorem{prop}[thm]{Proposition}

\newtheorem{lem}[thm]{Lemma}

\theoremstyle{definition}
\newtheorem{defn}[thm]{Definition}

\newtheorem{remark}[thm]{Remark}
\newtheorem{ex}[thm]{Example}
\newtheorem{notation}[thm]{Notation}

\numberwithin{equation}{section}
\numberwithin{thm}{section}

\title{Anchored symplectic embeddings}
\date{}
\author{Michael Hutchings, Agniva Roy, Morgan Weiler, and Yuan Yao}

\begin{document}
\maketitle
\begin{abstract}
Given two four-dimensional symplectic manifolds, together with knots in their boundaries, we define an ``anchored symplectic embedding'' to be a symplectic embedding, together with a two-dimensional symplectic cobordism between the knots (in the four-dimensional cobordism determined by the embedding). We use techniques from embedded contact homology to determine sharp quantitative critera for when anchored symplectic embeddings exist, for many examples of toric domains. In particular we find examples where ordinarily symplectic embeddings exist, but they cannot be upgraded to anchored symplectic embeddings unless one enlarges the target domain.
\end{abstract}

\input{intro.tex}

\input{ech.tex}

\input{proofs.tex}

\input{bibliography.tex}

\end{document}

%% file: intro.tex
\section{Introduction}

\subsection{Basic definitions}

Let $(X,\omega)$ and $(X',\omega')$ be compact symplectic manifolds with boundary of the same dimension. A {\em symplectic embedding\/} of $(X,\omega)$ into $(X',\omega')$ is a smooth embedding $\varphi:X\to X'$ such that $\varphi^*\omega'=\omega$. Since Gromov proved the celebrated nonsqueezing theorem \cite{gromov} in 1985, there has been much work studying when symplectic embeddings exist; see e.g.\ the surveys \cite{cghs,pnas,schlenk-survey}. 

In this paper we consider a ``relative'' version of this question in the four dimensional case. Suppose that $(X,\omega)$ and $(X',\omega')$ are compact symplectic four-manifolds with boundary $Y$ and $Y'$. Let $\gamma$ and $\gamma'$ be oriented knots in $Y$ and $Y'$ respectively.

\begin{defn}
\label{def:ase}
An {\em anchored symplectic embedding\/}
\[
(X,\omega,\gamma)\longrightarrow(X',\omega',\gamma')
\]
is a pair $(\varphi,\Sigma)$, where
\[
\varphi:(X,\omega)\longrightarrow(\op{int}(X'),\omega')
\]
is a symplectic embedding, and
\[
\Sigma\subset X'\setminus\varphi(\op{int}(X))
\]
is a smoothly embedded symplectic surface, which we call an ``anchor'', such that
\begin{equation}
\label{eqn:anchored}
\partial\Sigma = \gamma' - \varphi(\gamma).
\end{equation}
We require that $\Sigma$ is tranverse to $\partial X' \cup \varphi(\partial X)$.
\end{defn}

\begin{remark}
In the examples we consider, $\gamma$ will be a closed characteristic in $\partial X$, i.e.\ a closed leaf of the characteristic foliation $\Ker(\omega|_{T\partial X})$, and likewise $\gamma'$ will be a closed characteristic in $\partial X'$. In this case the surface $\Sigma$ is automatically transverse to $\partial X' \cup \varphi(\partial X)$ because it is symplectic.
\end{remark}

\begin{remark}
The definition of ``anchored symplectic embedding'' also makes sense for symplectic manifolds of higher dimension. However this case is less interesting, because as pointed out in \cite{eg}, there is an $h$-principle for symplectic submanifolds of codimension greater than two \cite[Thm.\ 12.1.1]{em}. This allows a smooth embedding $\Sigma$ satisfying \eqref{eqn:anchored} to be isotoped to a symplectic embedding by a $C^0$-small isotopy under mild hypotheses. See also Remark~\ref{rem:hprinciple} below.
\end{remark}

\begin{figure}
        \centering
        \begin{overpic}[scale=0.7]{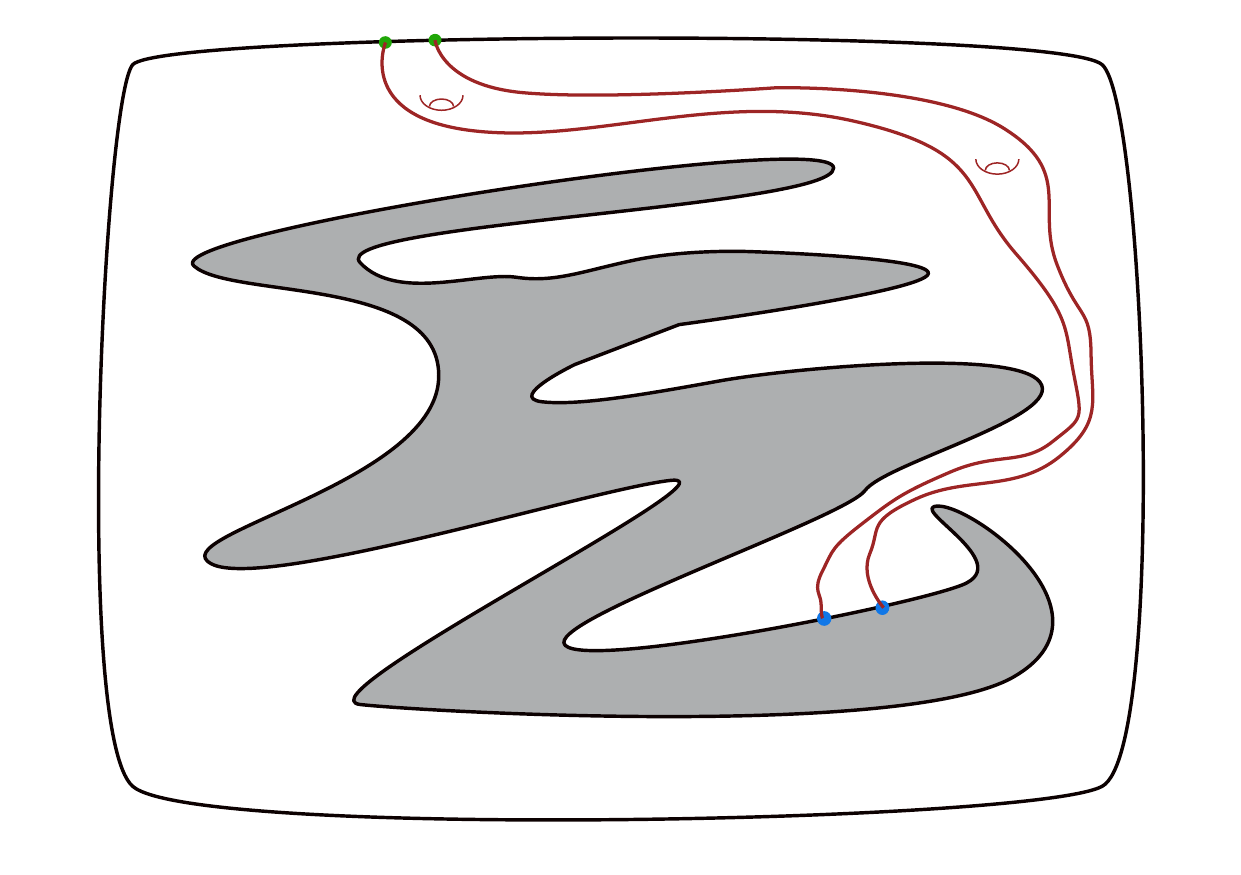}
        \put(5,5){$(X', \omega')$}
        \put(18,10){$\varphi(X, \omega)$}
        \put(21,7){{\color{blue}$\varphi(\gamma)$}}
        \put(8,18){{\color{Green}$\gamma'$}}
        \put(23,14){{\color{purple}$\Sigma$}}
        \end{overpic}
        \caption{A schematic of an anchored symplectic embedding $(\varphi,\Sigma) : (X,\omega,\gamma) \rightarrow (X',\omega',\gamma')$.}
                \label{fig:schematic}
        
    \end{figure}
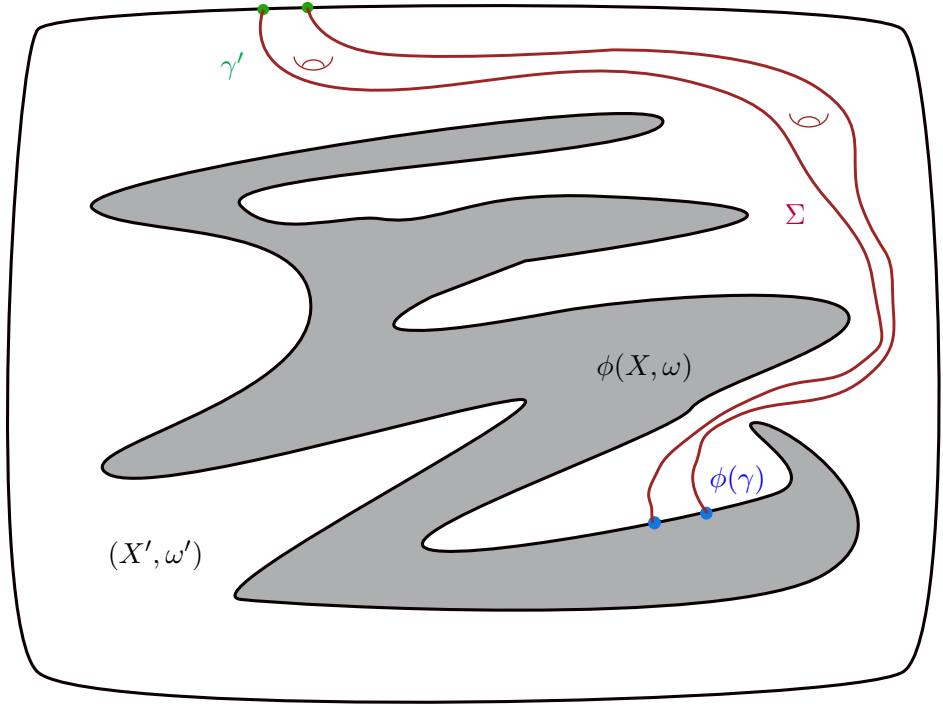

The goal of this paper is to study quantitative obstructions to anchored symplectic embeddings. In particular, we will see examples where there does not exist an anchored symplectic embedding $(X,\omega,\gamma)\to(X',\omega',\gamma')$, although there does exist a symplectic embedding $(X,\omega)\to(X',\omega')$, and moreover there exists an anchored symplectic embedding $(X,\omega,\gamma)\to(X',r\omega,\gamma')$ if $r>1$ is sufficiently large. 

\begin{remark}
The following are some related works on slightly different questions.

The paper \cite{eg} studies the notion of ``symplectic hat''; there one starts with $(X,\omega,\gamma)$ and looks for an anchored symplectic embedding as in Definition~\ref{def:ase}, except that $(X',\omega')$ is a closed symplectic manifold, and $\gamma'=\emptyset$. The emphasis is on topological rather than quantitative criteria for the existence of symplectic hats, and $(X',\omega')$ may be much larger than $(X,\omega)$. 

There is also an ``orthogonal'' story giving quantitative obstructions to Lagrangian cobordisms between Legendrian knots or between transverse knots; see e.g.\ \cite{datta,ds,st10,st17}. Topological obstructions to Lagrangian cobordisms between Legendrian knots have been previously studied using the Chekanov-Eliashberg DGA; see e.g.\ the survey \cite[\S6]{etnyreng}.

In addition, from the perspective of $C^0$ symplectic geometry, there have been a few rigidity results concerning how symplectic homeomorphisms may act on codimension two symplectic submanifolds; see e.g. \cite{buhov-ops}. The recent papers \cite{ hkho, hkho-2} on symplectic barriers can also be interpreted as how sympletomorphisms can act on symplectic submanifolds: they construct a grid of symplectic submanifolds that must intersect with the image of any large symplectic ball. This can be thought of as saying there does not exist an ambient symplectomorphism that stretches the grid so that its complement contains a large symplectic ball. 
\end{remark}


\subsection{Toric domains}

The examples of symplectic four-manifolds that we will consider are all ``toric domains'', which we now review.

\begin{defn}
\label{def:toricdomain}
Let $\Omega$ be a compact region in $\R^2_{\geq0}$. Assume that $0\in\op{int}(\Omega)$ and that $\partial\Omega$ consists of the line segment from $(0,0)$ to $(a,0)$ for some positive real number $a=a(\Omega)$, the line segment from $(0,0)$ to $(0,b)$ for some positive real number $b=b(\Omega)$, and a continuous curve $\partial_+\Omega$ in $\R^2_{\ge 0}$ from $(a,0)$ to $(0,b)$ which intersects the axes only at its endpoints. We define the \textit{toric domain}
\[
X_\Omega = \mu^{-1}(\Omega)
\]
where $\mu:\C^2\to\R^2_{\ge 0}$ is given by $(z_1,z_2)\mapsto\left(\pi|z_1|^2,\pi|z_2|^2\right)$. We equip $X_\Omega$ with the restriction of the standard symplectic form on $\R^4=\C^2$. Note that if the curve $\partial_+\Omega$ is smooth and transverse to the axes, then $\partial X_\Omega$ is a smooth hypersurface in $\R^4$, and in this case we say that $X_\Omega$ is {\em smooth\/}.

We say that $X_\Omega$ is a {\em convex toric domain\/}\footnote{A ``convex toric domain'' is not the same as a toric domain that is convex; see the discussion in \cite[\S2]{ghr}. Different and broader notions of ``convex toric domain'' are studied in \cite{Cristofaro-Gardiner19,Cristofaro-GardinerHolmMandiniPires}.} if the set
\[
\widehat{\Omega} = \{\mu\in\R^2\mid (|\mu_1|,|\mu_2|)\in\Omega\}
\]
is convex.
We say that $X_\Omega$ is a {\em concave toric domain\/} if the set $\R^2_{\ge 0}\setminus\Omega$ is convex. See Figure~\ref{fig:three graphs}.
\end{defn}

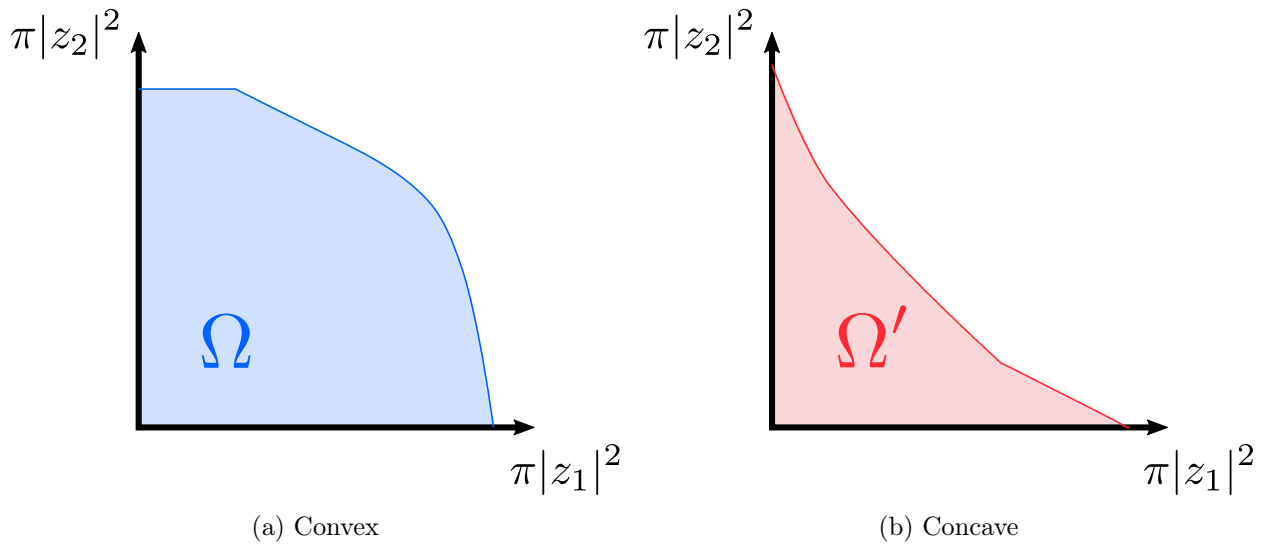
\begin{figure}
     \centering
     \begin{subfigure}[b]{0.49\textwidth}
         \centering
         \begin{overpic}[scale=0.7]{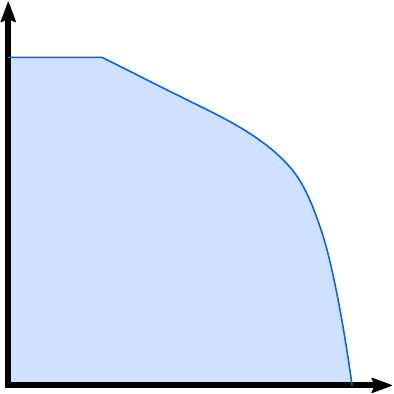}
         \put(2.75,2.75){$\Omega$}
\put(9.8,0){$\pi r_1^2$}
\put(-.4,10){$\pi r_2^2$}
        \end{overpic}
         \caption{Convex}
     \end{subfigure}
     \hfill
     \begin{subfigure}[b]{0.49\textwidth}
         \centering
         \begin{overpic}[scale=0.7]{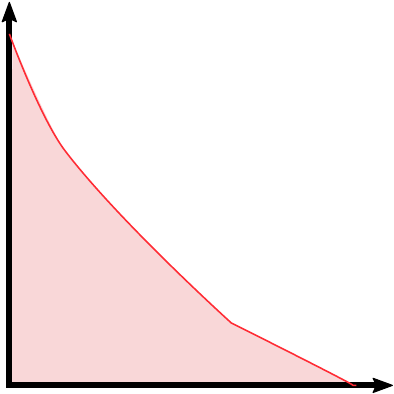}
         \put(2,2){$\Omega'$}
\put(9.8,0){$\pi r_1^2$}
\put(-.4,10){$\pi r_2^2$}
        \end{overpic}
         \caption{Concave}
     \end{subfigure}
        \caption{Regions $\Omega$ and $\Omega'$ in $\R^2_{\geq0}$ determining convex and concave toric domains respectively. Neither is smooth.}
        \label{fig:three graphs}
\end{figure}

\begin{ex}
The following are some basic examples of convex toric domains.
\begin{enumerate}[label=(\roman*)]
\item If $\Omega$ is the triangle with vertices $(0,0)$, $(r,0)$, and $(0,r)$, then $X_\Omega$ is the {\em ball\/}
\[
B^4(r)=\{z\in\C^2\mid \pi|z|^2\le r\}.
\]
\item More generally, if $\Omega$ is the triangle with vertices $(0,0)$, $(a,0)$, and $(0,b)$, then $X_\Omega$ is the \textit{ellipsoid}
\[
E(a,b) =\left\{(z_1,z_2)\in\C^2\;\bigg|\;\frac{\pi|z_1|^2}{a}+\frac{\pi|z_2|^2}{b}\leq1\right\}.
\]
\item If $\Omega$ is the rectangle $[0,a]\times[0,b]$, then $X_\Omega$ is the \textit{polydisk}
\[
P(a,b)=\left\{(z_1,z_2)\in\C^2\;\big|\;\pi|z_1|^2\leq a, \quad\pi|z_2|^2\leq b\right\}.
 \]
Note that the ellipsoid $E(a,b)$ is smooth, but the polydisk $P(a,b)$ is not smooth.
\end{enumerate}
\end{ex}

\begin{remark}
\label{rem:exceptionalorbits}
If $X_\Omega$ is a smooth toric domain, then there are two distinguished closed characteristics in $\partial X_\Omega$, given by the circles $(\pi|z_1|^2=a(\Omega),\; z_2=0)$ and $(z_1=0,\; \pi|z_2|^2=b(\Omega))$. We denote these closed characteristics by $e_{1,0}=e_{1,0}(X_\Omega)$ and $e_{0,1}=e_{0,1}(X_\Omega)$ respectively.
\end{remark}

\begin{remark}
\label{rem:inclusion}
If $\Omega\subset\op{int}(\Omega')$, then there is an anchored symplectic embedding
\[
(\varphi,\Sigma) : (X_\Omega,e_{1,0}) \longrightarrow (X_{\Omega'},e_{1,0}).
\]
Here $\varphi$ is given by the inclusion map $X_\Omega\to X_{\Omega'}$, while $\Sigma$ is the annulus
\[
\left\{z\in\C^2\;\big|\;a(\Omega) \le \pi|z_1|^2 \le a(\Omega'),\quad z_2=0\right\}.
\]
\end{remark}


\subsection{Statement of main results}

We now consider examples where the inclusion in Remark~\ref{rem:inclusion} is optimal.

\begin{thm}
\label{thm:polydiskball}
(proved in \S\ref{sec:convexproofs})
If $a>1$, then there exists an anchored symplectic embedding $(P(a,1),e_{1,0}) \rightarrow (B^4(c), e_{1,0})$ if and only if $c > a+1$, or equivalently $P(a,1) \subset \op{int}(B^4(c))$.
\end{thm}

\begin{remark}
The statement of Theorem~\ref{thm:polydiskball} makes sense, even though $P(a,1)$ is not smooth, because $e_{1,0}(P(a,1))$ is contained in the smooth part of $\partial P(a,1)$.
\end{remark}

\begin{remark}
\label{rem:folding}
If $1\le a \le 2$, then already the existence of a symplectic embedding $P(a,1)\to B^4(c)$ implies\footnote{See also \cite[Thm.\ 1]{ho} for a stronger result about Lagrangian embeddings which implies this.} that $c\ge a+1$, by \cite[Thm.\ 1.3]{Hutchings16}. 
If $a>2$, then better symplectic embeddings are possible. In particular, if $a>2$ then one can use symplectic folding \cite[Prop.\ 4.3.9]{schlenk-folding} to construct a symplectic embedding $\varphi:P(a,1)\to B^4(c)$ whenever\footnote{This lower bound on $c$ is known to be optimal when $2 \le a \le \frac{5+\sqrt{7}}{3}\approx 2.54$ by \cite{cn}. When $a>6$, better symplectic embeddings are possible using multiple symplectic folding \cite[\S4.3.2]{schlenk-folding}.} $c > 2 + a/2$.
\end{remark}

The following is a much larger family of examples (which however does not include Theorem~\ref{thm:polydiskball} as a special case).

\begin{thm}
\label{thm:convex1}
(proved in \S\ref{sec:convexproofs})
Let $X_{\Omega}$ and $X_{\Omega'}$ be convex toric domains in $\R^4$. Suppose that
\begin{equation}
\label{eqn:ab}
a(\Omega) > b(\Omega').
\end{equation}
If there exists an anchored symplectic embedding
\begin{equation}
\label{eqn:convex1}
(X_{\Omega}, e_{1,0}) \longrightarrow (X_{\Omega'}, e_{1,0})
\end{equation}
then $\Omega \subset \Omega'$. 
\end{thm}

\begin{remark}
The existence of an anchored symplectic embedding \eqref{eqn:convex1} forces $a(\Omega)<a(\Omega')$, for the simple reason that the anchor $\Sigma$ must have positive symplectic area, and this symplectic area would be $a(\Omega')-a(\Omega)$.
\end{remark}

\begin{remark}
There are many examples of pairs of convex toric domains $X_\Omega$ and $X_{\Omega'}$ for which a symplectic embedding $X_\Omega\to \op{int}(X_{\Omega'})$ exists, but an anchored symplectic embedding $(X_\Omega,e_{1,0})\to(X_{\Omega'},e_{1,0})$ is obstructed by Theorem~\ref{thm:convex1}, even though the positive area condition $a(\Omega)<a(\Omega')$ holds. For example, if $a>2$ then one can use symplectic folding as in \cite[\S4.3.2]{schlenk-folding} to construct a symplectic embedding of $P(a,1)$ into an arbitrarily small neighborhood of $P(1+a/2,2) \cap B(2+a/2)$. In particular there exists a symplectic embedding $P(8,2) \to E(11,7)$. However Theorem~\ref{thm:convex1} shows that such an embedding cannot be upgraded to an anchored symplectic embedding $(P(8,2),e_{1,0})\to(E(11,7),e_{1,0})$.
\end{remark}

Without the hypothesis \eqref{eqn:ab}, we can prove a similar result for ``2-anchored symplectic embeddings''. Suppose that $(X,\omega)$ and $(X',\omega')$ are compact symplectic four-manifolds with boundary $Y$ and $Y'$. Let $\gamma_1$ and $\gamma_2$ be disjoint knots in $Y$, and let $\gamma_1'$ and $\gamma_1'$ be disjoint knots in $Y'$.

\begin{defn}
\label{def:2ase}
A {\em 2-anchored symplectic embedding\/}
\[
(X,\omega,\gamma_1,\gamma_2)\longrightarrow(X',\omega',\gamma_1',\gamma_2')
\]
is a triple $(\varphi,\Sigma_1,\Sigma_2)$, where
\[
\varphi:(X,\omega)\longrightarrow(\op{int}(X'),\omega')
\]
is a symplectic embedding, and
\[
\Sigma_1,\Sigma_2\subset X'\setminus\varphi(\op{int}(X))
\]
are disjoint smoothly embedded symplectic surfaces such that
\[
\partial\Sigma_i = \gamma_i' - \varphi(\gamma_i).
\]
We also require that $\Sigma_i$ is transverse to $\partial X' \cup \varphi(\partial X)$.
\end{defn}

\begin{thm}
\label{thm:convex2}
(proved in \S\ref{sec:convexproofs})
Let $X_{\Omega}$ and $X_{\Omega'}$ be convex toric domains in $\R^4$. If there exists a 2-anchored symplectic embedding 
\[
(X_{\Omega}, e_{1,0},e_{0,1}) \longrightarrow (X_{\Omega'},e_{1,0},e_{0,1}),
\]
then $\Omega \subset \Omega'$.
\end{thm}

We also have an analogous result for concave toric domains:

\begin{thm}
\label{thm:concave2}
(proved in \S\ref{sec:concaveproofs})
Let $X_{\Omega}$ and $X_{\Omega'}$ be concave toric domains in $\R^4$. If there exists a 2-anchored symplectic embedding 
\[
(X_{\Omega}, e_{1,0},e_{0,1}) \longrightarrow (X_{\Omega'},e_{1,0},e_{0,1}),
\]
then $\Omega \subset \Omega'$.
\end{thm}

Note that Theorems~\ref{thm:convex2} and \ref{thm:concave2} are essentially sharp; in each theorem the $2$-anchored symplectic embedding in question exists if $\Omega\subset\op{int}(\Omega')$, similarly to Remark~\ref{rem:inclusion}.


\subsection{More about anchors}

To get a better understanding of the difference between a symplectic embedding and an anchored symplectic embedding, fix a symplectic embedding $\varphi:(X,\omega)\to(\op{int}(X'),\omega')$ and knots $\gamma\subset\partial X$ and $\gamma'\subset\partial X'$. When can the symplectic embedding $\varphi$ be upgraded to an anchored symplectic embedding? That is, when does there exist an embedded symplectic surface $\Sigma\subset W=X' \setminus \varphi(\op{int}(X)$ satisfying the requirements in Definition~\ref{def:ase}?

There are three basic necessary conditions. To describe them, let $H_2(W,\gamma',\varphi(\gamma))$ denote the set of 2-chains $Z$ in $W$ with $\partial X = \gamma' - \varphi(\gamma)$, modulo boundaries of 3-chains. This is an affine space over $H_2(W)$. An anchor $\Sigma$ as above determines a ``relative homology class'' $Z=[\Sigma]\in H_2(W,\gamma',\varphi(\gamma))$. Since $\Sigma$ is symplectic, we must have $\int_{Z}\omega' > 0$. In addition, the relative adjunction formula (see e.g.\ \cite[\S4.4]{Hutchings09}) determines the genus of $\Sigma$ in terms of the relative homology class $Z$, and this genus must be nonnegative. And, of course, there must exist a smoothly embedded surface in the relative homology class $Z$ of the correct genus.

However, the existence of a relative homology class $Z$ satisfying these conditions is not sufficient, as shown by the following simple example.

\begin{thm}
\label{thm:noanchor}
(proved in \S\ref{sec:convexproofs})
Let $X_\Omega, X_{\Omega'} \subset \R^4$ be convex toric domains with $\Omega\subset \op{int}(\Omega')$. Then the inclusion map $\imath:X_\Omega \to X_{\Omega'}$ can be upgraded to an anchored symplectic embedding
\[
(\imath, \Sigma) : (X_\Omega, e_{1,0}) \longrightarrow (X_{\Omega'}, e_{0,1}),
\]
if and only if
\begin{equation}
\label{eqn:bxy}
b(\Omega') > x_0 + y_0,
\end{equation}
where $(x_0,y_0)\in\partial\Omega$ is a point where the tangent line to $\partial\Omega$ has slope $-1$.
\end{thm}

\begin{remark}
\label{rem:eta}
The ``if'' part of Theorem~\ref{thm:noanchor} is proved as follows. Let $\eta$ be a smooth path in $\Omega' \setminus\op{int}(\Omega)$ which starts at $(0,b(\Omega'))$, where it is transverse to $\partial_+\Omega'$, and ends at $(a(\Omega),0)$, where it is transverse to $\partial_+\Omega$. The path $\eta$ lifts to an embedded cylinder $\Sigma\subset X_{\Omega'} \setminus \op{int}(X_\Omega)$ with $\partial\Sigma = e_{0,1}(\Omega') - e_{1,0}(\Omega)$, such that for each point $(x,y)$ in the interior of $\eta$, the intersection of $\Sigma$ with the 2-torus $\mu^{-1}(x,y)$ is a geodesic in the homology class $(1,1)$. The cylinder $\Sigma$ is symplectic with the correct orientation if and only if the function $x+y$ has negative derivative along the path $\eta$. (If $x+y$ is constant along the path, then the lifted cylinder is Lagrangian.) The existence of a path $\eta$ as above along which the function $x+y$ has negative derivative is equivalent to the inequality \eqref{eqn:bxy}. See Figure \ref{fig:eta}.
\end{remark}

\begin{figure}
\centering
\begin{overpic}[scale=0.7]{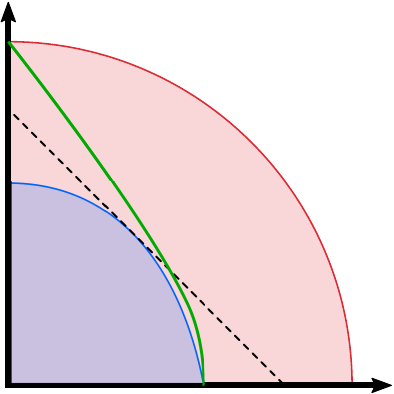}
\put(1.7,1.7){$\Omega$}
\put(4.5,4.5){$\Omega'$}
\put(2,6.5){\color{Green}\Large$\eta$}
\put(9.8,0){$\pi r_1^2$}
\put(-.4,10){$\pi r_2^2$}
        \end{overpic}
\caption{An example of regions $\Omega$ and $\Omega'$ (in blue and red, respectively) illustrating Remark \ref{rem:eta}. The tangent line relevant to (\ref{eqn:bxy}) is the dotted black line, and intersects the $y$-axis below $b(\Omega')$; a possible path $\eta$ of slope less than $-1$ is in green.}
\label{fig:eta}
\end{figure}
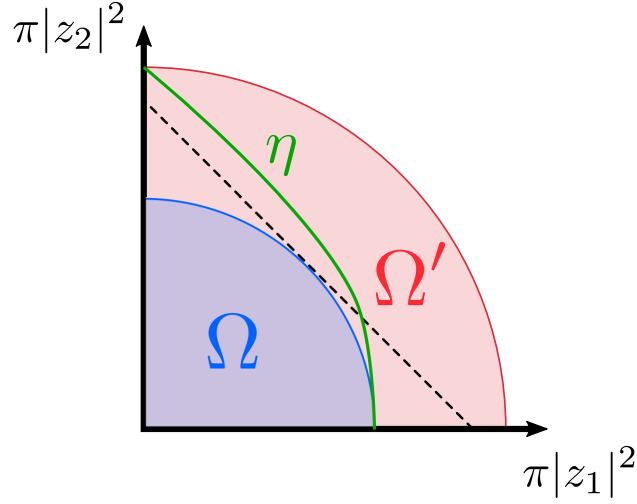

\begin{remark}
\label{rem:hprinciple}
In Remark~\ref{rem:eta}, suppose that $x+y$ does not have negative derivative along the path $\eta$, so that the cylinder $\Sigma$ is not symplectic. If we assume that $b(\Omega') > a(\Omega)$ (this is a weaker condition than the inequality \eqref{eqn:bxy}), then one can use h-principle arguments to show that $\Sigma$ has a $C^0$-small regular homotopy rel boundary to an immersed symplectic cylinder in $X_{\Omega'} \setminus \op{int}(X_\Omega)$. However Theorem~\ref{thm:noanchor} tells us that if condition \eqref{eqn:bxy} does not hold, then the self-intersections of this cylinder cannot be cancelled within $X_{\Omega'}\setminus X_{\Omega}$ while keeping the cylinder symplectic (although there does exist an embedded symplectic cylinder in $X_{\Omega'}$).
\end{remark}


\subsection{Idea of the proofs}

We will prove the main results in \S\ref{sec:proofs}, after assembling necessary background in \S\ref{sec:ech}. The basic principle is as follows. Let $(X,\omega)$ and $(X',\omega')$ be smooth star-shaped domains in $\R^4$. Given a symplectic embedding $\varphi: (X,\omega) \to (\op{int}(X'),\omega')$, we obtain a symplectic cobordism $W=X'\setminus\varphi(\op{int}(X))$. There is an associated cobordism map on embedded contact homology (ECH),
\begin{equation}
\label{eqn:PhiW}
\Phi = \Phi(W): ECH(\partial X') \longrightarrow ECH(\partial X),
\end{equation}
defined in \cite{HutchingsTaubes13}, which in this case is an isomorphism. For a suitable almost complex structure $J$ on the ``symplectic completion'' $\overline{W}$ of $W$, one can find a chain map $\phi$ inducing $\Phi$, such that whenever a coefficient of $\phi$ is nonzero, there is a corresponding ``broken $J$-holomorphic current'' in $\overline{W}$. Since any holomorphic curve has positive symplectic area, the existence of such a broken holomorphic current leads to an inequality involving the symplectic actions of the Reeb orbits to which its ends are asymptotic.

Now given an anchor $\Sigma$, after straightening $\Sigma$ near the boundary and ``completing'' it to an appropriate surface $\overline{\Sigma}$, we can choose $J$ such that $\overline{\Sigma}$ is holomorphic. Any other holomorphic curve in $\overline{W}$ must have positive intersections with $\overline{\Sigma}$. Combined with intersection theory developed in \cite{Hutchings16}, this leads to restrictions on which components of the chain map $\phi$ can be nonzero. (Compare \cite[Rem.\ 7.1]{calabi}.) These restrictions sharpen the inequalities described in the previous paragraph. 

The proof of Theorem \ref{thm:polydiskball} follows this strategy, studying how the chain map $\phi$ acts on the degree 4 part of the ECH chain complex. The proofs of Theorems~\ref{thm:convex1} and \ref{thm:convex2} for convex toric domains are more involved and require considering the ECH chain complex in arbitrarily high degrees. We do not need to understand the full chain complex, which is quite complicated (see \cite{Choi}), but we do prove a useful fact (Proposition~\ref{prop:eab}) asserting that certain chain complex generators are cycles representing nonzero classes in ECH. The hypotheses of these theorems, together with intersection theory as above, imply that the chain map $\phi$ preserves these generators. Then the fact that holomorphic curves have positive area translates into statements about the geometry of the toric domains, from which we conclude the inclusion statement. Theorem~\ref{thm:concave2}, for concave toric domains, is proved by a ``dual'' argument to the proof of Theorem~\ref{thm:convex2}. The proof of Theorem~\ref{thm:noanchor} is a variant of these arguments, roughly speaking studying a chain homotopy of ECH cobordism maps induced by a one-parameter family of almost complex structures.

We remark that instead of using ECH cobordism maps, we could instead find the holomorphic curves in $\overline{W}$ that we need using more ``elementary'' arguments; see \cite[Rem.\ 11]{altech}.

\paragraph{Acknowledgments.} M.H. was partially supported by NSF grant DMS-2005437. A.R. was partially supported by NSF grants DMS-2203312 and DMS-1907654. M.W. was partially supported by NSF grant DMS-2103245. Y.Y. was partially supported by ERC Starting Grant No. 851701. A.R. and Y.Y. are grateful to the Kylerec 2022 workshop. M.W. and Y.Y. also thank the 2023 Cornell Topology Festival. We further thank Ko Honda, Cagatay Kutluhan, Steven Sivek, John Etnyre, and Jen Hom for helpful conversations. We thank the anonymous referee for carefully reading the manuscript and pointing out a number of valuable corrections.

%% file: ech.tex
\section{ECH of perturbations of convex and concave toric domains}
\label{sec:ech}

In this section we prepare for the proofs of the main theorems.
In \S\ref{ss:ECHrev} we review what we need to know about embedded contact homology (ECH). In \S\ref{ss:TDrev} and \ref{ss:cECC} we combinatorially describe the ECH chain complex (not including the differential) for certain ``nice'' perturbations of convex and concave toric domains, along with additional information about the chain complex generators such as symplectic action and linking numbers. Finally, in \S\ref{ss:cob} and \S\ref{sec:cobordismspecial} we introduce what we need to know about ECH cobordism maps.


\subsection{Embedded contact homology}
\label{ss:ECHrev}

Let $Y$ be a closed oriented three-manifold (not necessarily connected). Let $\lambda$ be a contact form on $Y$, and assume that $\lambda$ is nondegenerate (see below). To simplify the discussion\footnote{See \cite{Hutchingslec} for the definition of ECH without the homological assumption \eqref{eqn:homologysphere}.}, assume that
\begin{equation}
\label{eqn:homologysphere}
H_1(Y)=H_2(Y)=0.
\end{equation}
(In the cases relevant to the proofs of the main results, $Y$ will be the boundary of a smooth toric domain and thus diffeomorphic to $S^3$.) We now review how to define the embedded contact homology of $(Y,\lambda)$ with $\Z/2$ coefficients\footnote{It is also possible to define ECH with $\Z$ coefficients \cite[\S9]{HutchingsTaubes09}.}, which we denote by $ECH_*(Y,\lambda)$. This is a $\Z$-graded $\Z/2$-module (the definition of the $\Z$-grading will use the homological assumption \eqref{eqn:homologysphere}).

\paragraph{Contact geometry.}
The contact form $\lambda$ determines the contact structure $\xi=\ker\lambda$, as well as the Reeb vector field $R$ characterized by
\[
d\lambda(R,\cdot)=0, \quad \quad \lambda(R)=1.
\]
A \textit{Reeb orbit} is a smooth map
\begin{equation}
\label{eqn:Reeborbit}
\gamma:\R/T\Z\longrightarrow Y \quad \mbox{ with } \quad \gamma'(t)=R(\gamma(t)).
\end{equation}
We declare two Reeb orbits to be equivalent if they differ by reparametrization of the domain. We say that the Reeb orbit $\gamma$ is {\em simple\/} if the map $\gamma$ is an embedding.

Let $\{\psi_t:Y\circlearrowleft\}_{t\in\R}$ denote the flow of the Reeb vector field $R$. Let $\gamma$ be a Reeb orbit as above. The derivative of the time $t$ Reeb flow restricts to a symplectic linear map
\begin{equation}
\label{eqn:dpsi}
d\psi_t:(\xi_{\gamma(0)},d\lambda) \longrightarrow (\xi_{\gamma(t)},d\lambda),
\end{equation}
and $d\psi_T$ is called the \textit{linearized return map} of $\gamma$. We say that $\gamma$ is {\em nondegenerate\/} if $1$ is not an eigenvalue of $d\psi_T$. In this case we say that $\gamma$ is {\em elliptic\/} if the eigenvalues of $d\psi_T$ are on the unit circle, {\em positive hyperbolic\/} if the eigenvalues of $d\psi_T$ are positive, and {\em negative hyperbolic\/} if the eigenvalues of $d\psi_T$ are negative. We say that the contact form $\lambda$ is {\em nondegenerate\/} if all Reeb orbits are nondegenerate, and we assume below that this is the case.

If $\gamma$ is a Reeb orbit $\gamma$, its {\em symplectic action\/} $\mathcal{A}(\gamma)$ is the integral of the contact form
\[
\mathcal{A}(\gamma) = \int_{\gamma}\lambda,
\]
or equivalently the period $T$ in \eqref{eqn:Reeborbit}.

\paragraph{The chain module.}

\begin{defn}
An \textit{orbit set} is a finite set of pairs $\alpha=\{(\alpha_i,m_i)\}$, where the $\alpha_i$ are distinct simple Reeb orbits and $m_i\in\Z_{>0}$. An {\em ECH generator\/} is an orbit set as above such that $m_i=1$ whenever $\alpha_i$ is hyperbolic.
\end{defn}

We sometimes use the multiplicative notation
\[
\prod_i\alpha_i^{m_i}  \longleftrightarrow  \{(\alpha_i,m_i)\}.
\]
We define $ECC_*(Y,\lambda)$ to be the free $\Z/2$-module generated by the ECH generators.

The module $ECC_*(Y,\lambda)$ has a $\Z$-grading by the ECH index, which is defined as follows.

\begin{defn}
If $\alpha$ is an ECH generator (or more generally an orbit set), its {\em ECH index\/} is defined by
\[
I(\alpha) = c_\tau(\alpha)+Q_\tau(\alpha)+\op{CZ}_\tau^I(\alpha)\in\Z
\]
where 
\begin{itemize}
\item
$\tau$ is a symplectic trivialization of $\xi$ over each of the simple orbits $\alpha_i$ in $\alpha$.
\item
$c_\tau$ is the \textit{relative first Chern number}, defined as follows. For each $i$, since $H_1(Y)=0$, we can choose a compact oriented surface $\Sigma_i$ in $Y$ with $\partial\Sigma_i=\alpha_i$. Then $c_\tau(\alpha_i)=c_1(\xi|_{\Sigma_i},\tau)$, where the right hand side is the signed count of zeroes of a generic section of $\xi|_{\Sigma_i}$ which over $\alpha_i=\partial\Sigma_i$ is constant and nonvanishing with respect to $\tau$. This count does not depend on the choice of $\Sigma_i$ since $H_2(Y)=0$. Finally,
\begin{equation}
\label{eqn:ctaulinear}
c_\tau(\alpha) = \sum_i m_i c_\tau(\alpha_i)\in\Z.
\end{equation}
\item
$Q_\tau$ is the \textit{relative intersection pairing} defined by
\begin{equation}
\label{eqn:Qtaubilinear}
Q_\tau(\alpha) = \sum_i m_i^2 Q_\tau(\alpha_i)+\sum_{i\neq j} m_im_j \ell(\alpha_i,\alpha_j)\in\Z.
\end{equation}
Here if $i\neq j$ then $\ell(\alpha_i,\alpha_j)$ denotes the linking number of $\alpha_i$ with $\alpha_j$, while $Q_\tau(\alpha_i)$ denotes the linking number of $\alpha_i$ with a pushoff of $\alpha_i$ via the framing $\tau$.
\item
We define
\[
\op{CZ}_\tau^I(\alpha) = \sum_i\sum_{k=1}^{m_i}\op{CZ}_\tau(\alpha_i^k)\in\Z.
\]
Here $\alpha_i^k$ denotes the $k^{th}$ iterate of $\alpha_i$, and $\op{CZ}_\tau$ denotes the Conley-Zehnder index; see e.g.\ the review in \cite[\S3.2]{Hutchingslec}.
\end{itemize}
\end{defn}

The ECH index does not depend on the choice of $\tau$, even though the individual terms in it do; see e.g.\ \cite[\S2.8]{Hutchings09}. 

\paragraph{The differential.}

\begin{defn}
\label{def:lambdacompatible}
An almost complex structure $J$ on $\R\times Y$ is {\em $\lambda$-compatible\/} if:
\begin{itemize}
\item
$J\partial_s=R$, where $s$ denotes the $\R$ coordinate on $\R\times Y$;
\item
$J$ maps the contact structure $\xi$ to itself, rotating positively with respect to $d\lambda$; and
\item
$J$ is invariant under translation of the $\R$ factor on $\R\times Y$.
\end{itemize}
\end{defn}

Fix a $\lambda$-compatible almost complex structure $J$. We consider $J$-holomorphic curves
\[
u:(C,j) \longrightarrow (\R\times Y,J)
\]
where the domain is a punctured compact Riemann surface. We assume that for each puncture, there exists a Reeb orbit $\gamma$, such that in a neighborhood of the puncture, $u$ is asymptotic to $\R\times\gamma$ as either $s\to +\infty$ (in which case we say that this is a ``positive puncture'') or $s\to-\infty$ (a ``negative puncture''). If $u$ is somewhere injective, then $u$ is determined by its image in $\R\times Y$, which by abuse of notation we still denote by $C$.

\begin{defn}
A {\em $J$-holomorphic current\/} is a finite sum $\mathcal{C}=\sum_kd_kC_k$ where the $C_k$ are distinct somewhere injective $J$-holomorphic curves as above, and the $d_k$ are positive integers. If $\alpha$ and $\beta$ are orbit sets, we define $\mathcal{M}^J(\alpha,\beta)$ to be the set of $J$-holomorphic currents $\mathcal{C}$ such that $\lim_{s\to+\infty}(\mathcal{C}\cap(\{s\}\times Y))=\alpha$ and $\lim_{s\to-\infty}(\mathcal{C}\cap(\{s\}\times Y))=\beta$ as currents\footnote{This means, for example, that if the pair $(\alpha_i,m_i)$ appears in $\alpha$, then some of the curves $C_k$ have positive punctures asymptotic to covers of $\alpha_i$, say of multiplicities $q_{i,k,l}$, and we have $\sum_kd_k\sum_lq_{i,k,l}=m_i$.}.
\end{defn}

Since $J$ is $\R$-invariant, $\R$ acts on $\mathcal{M}^J(\alpha,\beta)$ by translation of the $\R$ coordinate on $\R\times Y$. By \cite[Prop.\ 3.7]{Hutchingslec}, if $J$ is generic and $I(\alpha)-I(\beta)=1$, then $\mathcal{M}^J(\alpha,\beta)/\R$ is a finite set; moreover, for a given ECH generator $\alpha$, there are only finitely many ECH generators $\beta$ with $I(\alpha)-I(\beta)=1$ for which $\mathcal{M}^J(\alpha,\beta)$ is nonempty. We can then define the differential
\[
\partial_J: ECC_*(Y,\lambda) \longrightarrow ECC_{*-1}(Y,\lambda)
\]
as follows. If $\alpha$ is an ECH generator, then
\[
\partial_J\alpha = \sum_{I(\alpha)-I(\beta)=1} \#_{\mathbb{Z}_2}\frac{\mathcal{M}^J(\alpha,\beta)}{\R}\cdot\beta.
\]
Here $\beta$ is an ECH generator, and $\#_{\mathbb{Z}_2}$ denotes the mod 2 count.

It is shown in \cite{obg1} that $\partial_J^2=0$. We denote the homology of the chain complex $(ECC_*(Y,\lambda),\partial_J)$ by $ECH_*(Y,\lambda)$ or $ECH_*(Y,\xi)$. It is shown in \cite{Taubes10} that this homology is canonically isomorphic to a version of Seiberg-Witten Floer cohomology depending only on $Y$ and $\xi$ and not on $\lambda$ or $J$.

\paragraph{ECH of $S^3$.} As explained in \cite[\S3.7]{Hutchingslec}, we have
\[
ECH_*(S^3,\xi_{\op{std}},J)=\begin{cases}\Z_2&\text{ if }*=0,2,4,\dots\\0&\text{ else.}
\end{cases}
\]
Here $\xi_{\op{std}}$ denotes the standard tight contact structure on $S^3$.

\paragraph{Action filtration.}
Let $\alpha=\{\alpha_i,m_i\}$ denote an orbit set. We define its {\em symplectic action\/} to be
\[
\mathcal{A}(\mathcal{\alpha}) = \sum_i m_i \int_{\alpha_i} \lambda = \sum_im_i\A(\alpha_i).
\]
Given $L\in\R$, let $ECC_*^L(Y,\lambda)$ denote the subset of $ECC(Y,\lambda)$ spanned by ECH generators $\alpha$ with $\mathcal{A}(\alpha) < L$. If $J$ is a $\lambda$-compatible almost complex structure, then it follows from the second bullet in Definition~\ref{def:lambdacompatible} that if $\mathcal{M}^J(\alpha,\beta)\neq\emptyset$ then $\mathcal{A}(\alpha)\ge \mathcal{A}(\beta)$. Consequently $(ECC^L_*(Y,\lambda),\partial_J)$ is a subcomplex of $(ECC_*(Y,\lambda),\partial_J)$. The homology of this subcomplex is the {\em filtered ECH} which we denote by $ECH^L_*(Y,\lambda)$. It is shown in \cite[Thm.\ 1.3]{HutchingsTaubes13} that filtered ECH does not depend on $J$ (although it does depend on $\lambda$), and furthermore the maps
\[
ECH^L(Y,\lambda) \longrightarrow ECH^{L'}(Y,\lambda)
\]
for $L<L'$ and
\[
ECH^L(Y,\lambda) \longrightarrow ECH(Y,\xi)
\]
induced by inclusion of chain complexes are also independent of $J$.

\paragraph{$J_0$ index}

There is an important variant of the ECH index $I$, denoted by $J_0$ (not an almost complex structure). If $\alpha=\{(\alpha_i,m_i)\}$ is an ECH generator, we define
\[
J_0(\alpha) = I(\alpha) - 2c_\tau(\alpha) - \sum_i\op{CZ}_\tau(\alpha_i^{m_i}) \in \Z.
\]
According to \cite[Prop.\ 3.2]{Hutchings16}, the $J_0$ index bounds topological complexity of holomorphic curves as follows. Let $J$ be a $\lambda$-compatible almost complex structure, let $\alpha=\{(\alpha_i,m_i)\}$ and $\beta=\{(\beta_j,n_j)\}$ be ECH generators, and let $C\in\mathcal{M}^J(\alpha,\beta)$ be somewhere injective with connected domain. Then
\begin{equation}
\label{eqn:J0bound}
2g(C) - 2 + \sum_i\left(2n_i^+-1\right) + \sum_j\left(2n_j^--1\right) \le J_0(\alpha) - J_0(\beta).
\end{equation}
Here $g(C)$ denotes the genus of $C$, while $n_i^+$ denotes the number of positive punctures of $C$ asymptotic to covers of $\alpha_i$, and $n_j^-$ denotes the number of negative punctures of $C$ asymptotic to covers of $\beta_j$.


\subsection{Reeb dynamics on the boundary of a toric domain}
\label{ss:TDrev}

We now discuss the Reeb dynamics on the boundary of a toric domain. The following is a consolidation and review of material from \cite[\S3.2]{concave} which discusses concave toric domains and \cite[\S5]{Hutchings16} which discusses convex toric domains, with updated notational conventions from \cite{gh}.

Let $X_\Omega\subset\R^4$ be a smooth toric domain as in Definition~\ref{def:toricdomain}, and assume that $\partial_+\Omega$ is transverse to the radial vector field on $\R^2$ (which holds for example for convex and concave toric domains). Then $\partial X_\Omega$ is a star-shaped hypersurface in $\R^4$. As such, the standard Liouville form
\begin{equation}
\label{eqn:lambda0}
\lambda_0 = \frac{1}{2}\sum_{i=1}^2\left(x_i\,dy_i - y_i\,dx_i\right)
\end{equation}
restricts to a contact form on $\partial X_\Omega$.

As in Remark~\ref{rem:exceptionalorbits}, there are two distinguished Reeb orbits $e_{1,0}$ and $e_{0,1}$ in $\partial X_\Omega$, where $z_2=0$ and $z_1=0$ respectively. We now discuss the Reeb dynamics on the rest of $\partial X_\Omega$ where $z_1,z_2\neq 0$. Here we use coordinates $(r_1,\theta_1,r_2,\theta_2)$ where $z_1=r_1e^{i\theta_1}$ and $z_2=r_2e^{i\theta_2}$.

Let $(x,y)\in\partial\Omega$ with $x,y>0$. Let $(a,b)$ be an outward normal vector to $\partial\Omega$ at $(x,y)$. On the two-torus $\mu^{-1}(x,y)$, the Reeb vector field is given by
\begin{equation}
\label{eqn:toricReeb}
R = \frac{2\pi}{ax+by}\left(a\frac{\partial}{\partial\theta_1}+b\frac{\partial}{\partial\theta_2}\right).
\end{equation}
(See \cite[\S2.2]{gh} for more general computations.) Using \eqref{eqn:toricReeb} we obtain the following information about the Reeb orbits in $\partial X_\Omega$.

\begin{description}
\item[Simple orbits.]
If $a/b\in\Q\cup\{\infty\}$, then the torus $\mu^{-1}(x,y)$ is foliated by Reeb orbits (and all simple Reeb orbits in $\partial X_\Omega$ other than $e_{1,0}$ and $e_{0,1}$ arise this way). In this case, our convention is to rescale the normal vector $(a,b)$ so that $a,b$ are relatively prime nonnegative integers.
\item[Symplectic action.]
In the above situation, if $\gamma$ is a simple Reeb orbit in $\mu^{-1}(x,y)$, then its symplectic action is
\begin{equation}
\label{eqn:generalaction}
\mathcal{A}(\gamma) = ax+by.
\end{equation}
We also have
\begin{equation}
\label{eqn:exceptionalaction}
\mathcal{A}(e_{1,0})=a(\Omega), \quad\quad \mathcal{A}(e_{0,1})=b(\Omega).
\end{equation}
\item[Linking numbers.]
The linking numbers of simple Reeb orbits in $\partial X_\Omega \simeq S^3$ are given as follows. First,
$e_{1,0}$ and $e_{0,1}$ form a Hopf link, and in particular
\begin{equation}
\label{eqn:Hopflink}
\ell(e_{1,0},e_{0,1})=1.
\end{equation}
Second, if $\gamma$ is a simple Reeb orbit distinct from $e_{1,0}$ and $e_{0,1}$, and if $(a,b)$ is the associated integer vector as above, then as one traverses $\gamma$, the coordinate $\theta_1$ winds $a$ times and the coordinate $\theta_2$ winds $b$ times. Consequently
\begin{equation}
\label{eqn:exceptionallinking}
\begin{split}
\ell(\gamma,e_{1,0}) &= b,\\
\ell(\gamma,e_{0,1}) &= a,
\end{split}
\end{equation}
Finally, let $\gamma'$ be a simple Reeb orbit distinct from $e_{1,0}$, $e_{0,1}$, and $\gamma$, and let $(a',b')$ be the associated integer vector as above. Orient the curve $\partial_+\Omega$ from $(a,0)$ to $(0,b)$, and suppose that $\mu(\gamma)$ precedes $\mu(\gamma')$ along this curve or that $\mu(\gamma)=\mu(\gamma')$. Then we can homotope $\gamma$ to $e_{1,0}^a$ and $\gamma'$ to $e_{0,1}^{b'}$ without crossing, so
\begin{equation}
\label{eqn:generallinking}
\ell(\gamma,\gamma')=ab'.
\end{equation}
\item[Relative first Chern number.]
We can find a section $s$ of $\xi$ over $\partial X_\Omega$ such that $s^{-1}(0)=e_{1,0}\cup e_{0,1}$, and $s$ takes values in $\op{Ker}(d\mu)$. The section $s$ determines a trivialization $\tau$ of $\xi$ over all simple Reeb orbits other than $e_{1,0}$ and $e_{0,1}$.  If $\gamma$ is such a Reeb orbit with associated integer vector $(a,b)$, then it follows from \eqref{eqn:exceptionallinking} after an orientation check that
\begin{equation}
\label{eqn:generaltrivialization}
c_\tau(\gamma) = a+b.
\end{equation}
\end{description}

To describe the ECH of $\partial X_\Omega$ at the chain level, we need to perturb $X_\Omega$ so that the contact form on the boundary becomes nondegenerate. We now describe a ``nice'' way to do so for convex toric domains.

Suppose that $X_\Omega\subset\R^4$ is a convex toric domain. If $(a,b)\in\Z^2_{\ge 0}$, define
\begin{equation}
\label{eqn:Omeganorm}
\|(a,b)\|_\Omega^* = \max\{ax+by \mid (x,y)\in\Omega\}.
\end{equation}
Note that the maximum in \eqref{eqn:Omeganorm} is realized by a point $(x,y)\in\partial_+\Omega$ where $(a,b)$ is an outward normal vector.

\begin{lem}
\label{lem:Lniceconvex}
Let $X_\Omega\subset\R^4$ be a convex toric domain and let $L>\max(a(\Omega),b(\Omega))$. Then there exists a smooth star-shaped domain $X\subset\R^4$ with the following properties:
\begin{itemize}
\item
The $C^0$ distance between $\partial X_\Omega$ and $\partial X$ is at most $L^{-1}$.
\item
The contact form ${\lambda_0}|_{\partial X}$ is nondegenerate.
\item
The simple Reeb orbits in $\partial X$ with symplectic action less than $L$ consist of, for each pair of relatively prime nonnegative integers $(a,b)$ with $\|(a,b)\|_{\Omega}^*<L$, an elliptic simple Reeb orbit $e_{a,b}$, and a positive hyperbolic simple Reeb orbit $h_{a,b}$ when $a,b>0$.
\item
We can arrange that either $X\subset X_\Omega$ or $X_\Omega\subset X$, and that this inclusion upgrades to a 2-anchored symplectic embedding $(X,e_{1,0},e_{0,1})\to (X_\Omega,e_{1,0},e_{0,1})$ or vice versa.
\item
If $a,b$ are relatively prime nonnegative integers with $\|(a,b)\|_{\Omega}^*<L$, and if $\gamma_{a,b}$ denotes either $e_{a,b}$ or $h_{a,b}$ when $a,b>0$, then
\begin{equation}
\label{eqn:Lniceconvexaction}
\left|\mathcal{A}(\gamma_{a,b}) - \|(a,b)\|_{\Omega}^*\right| < L^{-1}.
\end{equation}
\item
If $(a',b')$ is another pair of relatively prime nonnegative integers with $\|(a',b')\|_\Omega^*<L$, and if $\gamma_{a',b'}$ is distinct from $\gamma_{a,b}$, then the linking number in $\partial X \simeq S^3$ is
\begin{equation}
\label{eqn:Lniceconvexlinking}
\ell(\gamma_{a,b},\gamma_{a',b'}) = \max(ab',a'b).
\end{equation}
\item
There is a trivialization $\tau$ of $\xi$ over all of the simple Reeb orbits with symplectic action less than $L$ such that
\begin{align}
\label{eqn:ctauconvex}
c_\tau(\gamma_{a,b}) &= a+b,\\
\label{eqn:Qtauconvex}
Q_\tau(\gamma_{a,b}) &= ab,\\
\label{eqn:ctaueabconvex}
\op{CZ}_\tau(e_{a,b}^m) &= 1,\quad\quad \mbox{\rm if} \quad m\|(a,b)\|_\Omega^*<L,\\
\label{eqn:ctauhab}
\op{CZ}_\tau(h_{a,b}) &= 0.
\end{align}
\end{itemize}
\end{lem}

\begin{proof}[Proof of Lemma~\ref{lem:Lniceconvex}.]
To start, by a $C^0$-small perturbation of $\partial_+\Omega$, we can arrange that (i) $\partial_+\Omega$ is smooth; (ii) $\partial_+\Omega$ is strictly convex; and (iii) $\partial_+\Omega$ is nearly perpendicular to the axes. More precisely, there is a small irrational $\epsilon>0$ such that where $\partial_+\Omega$ meets the $y$ axis, its slope is $-\epsilon$, and where $\partial_+\Omega$ meets the $x$ axis, its slope is $-\epsilon^{-1}$. In particular, for each pair of relatively prime positive integers $(a,b)$ with $\|(a,b)\|_\Omega^*<L$, there is a unique point $(x,y)\in\partial_+\Omega$ at which $(a,b)$ is an outward normal vector to $\partial_+\Omega$.

Given $(x,y)$ as above, since $\partial_+\Omega$ is strictly convex, the circle of Reeb orbits in $\mu^{-1}(x,y)$ is Morse-Bott. Similarly\footnote{The picture in \cite[\S3.1]{HutchingsSullivan05} corresponds to a concave toric domain, and for our convex case the directions of the arrows should be reversed.} to \cite[\S3.1]{HutchingsSullivan05} (see \cite[\S3.1]{bourgeois} for a more general situation), the contact form can be perturbed in a neighborhood of $\mu^{-1}(x,y)$ (which corresponds to perturbing $X_\Omega$), so that this circle of Reeb orbits is reduced to two nondegenerate simple Reeb orbits: an elliptic orbit $e_{a,b}$, for which the linearized return map is a slight positive rotation, and a positive hyperbolic orbit $h_{a,b}$. The perturbation may also create new Reeb orbits of action greater than $L$. We can further perturb the contact form to arrange that the Reeb orbits with action greater than $L$ are nondegenerate\footnote{This is not actually necessary to define the filtered embedded contact homology $ECH^L$.}.

The above implies the first four bullet points in the lemma. The action estimate \eqref{eqn:Lniceconvexaction} now follows from \eqref{eqn:generalaction} and \eqref{eqn:exceptionalaction}. The linking number formula \eqref{eqn:Lniceconvexlinking} follows from \eqref{eqn:Hopflink}, \eqref{eqn:exceptionallinking}, and \eqref{eqn:generallinking}.

To prove the last bullet point, over the simple Reeb orbits $e_{a,b}$ and $h_{a,b}$ with $a,b>0$, choose the trivialization $\tau$ as in \eqref{eqn:generaltrivialization}. Then for $a,b>0$, equation \eqref{eqn:ctauconvex} follows from \eqref{eqn:generaltrivialization}, equation \eqref{eqn:Qtauconvex} follows similarly to \eqref{eqn:Lniceconvexlinking}, and equations \eqref{eqn:ctaueabconvex} and \eqref{eqn:ctauhab} follow from the definition of the Conley-Zehner index in e.g.\ \cite[\S3.2]{Hutchingslec}. The trivialization $\tau$ has an extension over $e_{1,0}$ and $e_{0,1}$ satisfying equations \eqref{eqn:ctauconvex}, \eqref{eqn:Qtauconvex}, and \eqref{eqn:ctaueabconvex}, as in \cite[\S3.7]{Hutchingslec}.
\end{proof}

A slight modification of the above lemma holds for concave toric domains. To state it, if $X_\Omega\subset\R^4$ is a concave toric domain, and if $(a,b)\in\Z^2_{\ge 0}$, define
\[
[(a,b)]_\Omega = \min\{ax+by \mid (x,y)\in\partial_+\Omega\}.
\]

\begin{lem}
\label{lem:Lniceconcave}
Let $X_\Omega$ be a concave toric domain and let $L>\max(a(\Omega),b(\Omega))$. Then there exists a smooth star-shaped domain $X\subset\R^4$ such that:
\begin{itemize}
\item
The first five bullet points in Lemma~\ref{lem:Lniceconvex} hold, with $\|\cdot\|_\Omega^*$ replaced by $[\cdot]_\Omega$.
\item
If $(a,b)$ and $(a',b')$ are pairs of relatively prime nonnegative integers with $[(a,b)]_\Omega, [(a',b')]_\Omega <L$, and if $\gamma_{a',b'}$ is distinct from $\gamma_{a,b}$, then the linking number in $\partial X \simeq S^3$ is
\[
\ell(\gamma_{a,b},\gamma_{a',b'}) = \min(ab',a'b).
\]
\item
There is a trivialization $\tau$ of $\xi$ over all of the simple Reeb orbits with symplectic action less than $L$ such that
\[
\begin{split}
c_\tau(\gamma_{a,b}) &= a+b,\\
Q_\tau(\gamma_{a,b}) &= ab,\\
\op{CZ}_\tau(e_{a,b}^m) &= -1,\quad\quad \mbox{\rm if} \quad m\|(a,b)\|_\Omega^*<L \quad \mbox{and} \quad a,b>0,\\
\op{CZ}_\tau(h_{a,b}) &= 0,\\
\op{CZ}_\tau(e_{1,0}), \op{CZ}_\tau(e_{0,1}) &> L.
\end{split}
\]
\end{itemize}
\end{lem}

\begin{proof}
To start, by a $C^0$-small perturbation of $\partial_+\Omega$, we can arrange that (i) $\partial_+\Omega$ is smooth; (ii) $\partial_+\Omega$ is strictly concave; and (iii) $\partial_+\Omega$ is nearly tangent to the axes. More precisely, there is a small irrational $\epsilon>0$ such that where $\partial_+\Omega$ meets the $y$ axis, its slope is $-\epsilon^{-1}$, and where $\partial_+\Omega$ meets the $x$ axis, its slope is $-\epsilon$. The rest of the argument follows the proof of Lemma~\ref{lem:Lniceconvex}.
\end{proof}

\begin{remark}
We are using a different notational convention from \cite{concave,Hutchings16}; a Reeb orbit $e_{a,b}$ or $h_{a,b}$ here corresponds to $e_{b,a}$ or $h_{b,a}$ in those references.
\end{remark}


\subsection{Combinatorial ECH generators for convex and concave toric domains}
\label{ss:cECC}

We now review how to combinatorially describe the ECH generators, their ECH indices, and their approximate symplectic actions, for ``nice'' perturbations of convex and concave toric domains. This is based on \cite{Hutchings16} and \cite{concave} with some minor notational changes.

\subsubsection{Convex toric domains}

\begin{defn}
\cite[Def.\ 1.9]{Hutchings16}
A {\em convex integral path\/} is a path $\Lambda$ in the plane such that:
\begin{itemize}
\item
The endpoints of $\Lambda$ are $(0,y(\Lambda))$ and $(x(\Lambda),0)$ where $x(\Lambda)$ and $y(\Lambda)$ are nonnegative integers.
\item
$\Lambda$ is the graph of a piecewise linear concave function $f:[0,x(\Lambda)]\to[0,y(\Lambda)]$ with $f'(0)\le 0$, possibly together with a vertical line segment at the right.
\item
The {\em vertices\/} of $\Lambda$ (the points at which its slope changes, and the endpoints) are lattice points.
\end{itemize}
\end{defn}

\begin{notation}
If $v$ is an {\em edge\/} of a convex integral path (a line segment between consecutive vertices), then the vector from the upper left endpoint to the lower right endpoint of $v$ has the form $(b,-a)$ where $a,b$ are nonnegative integers. Write $v^\perp=(a,b)$, and define the {\em multiplicity\/} of $v$, which we denote by $m(v)$, to be the greatest common divisor of $a$ and $b$.
\end{notation}

\begin{defn}
\label{def:Omegaaction}
If $X_\Omega$ is a convex toric domain and $\Lambda$ is a convex integral path, define the {\em $\Omega$-action\/} of $\Lambda$ to be
\[
\mathcal{A}_{X_\Omega}(\Lambda) = \mathcal{A}_\Omega(\Lambda) = \sum_{v\in\op{Edges}(\Lambda)}\big\|v^\perp\big\|_\Omega^*.
\]
\end{defn}

\begin{defn}
\cite[Def.\ 1.10]{Hutchings16}
A {\em convex generator\/} is a convex integral path $\Lambda$, together with a labeling of each edge by `$e$' or `$h$' (we omit the labeling from the notation). Horizontal and vertical edges are required to be labeled `$e$'.
\end{defn}

\begin{notation}
If $\Lambda$ is a convex generator, let $h(\Lambda)$ denote the number of edges that are labeled `$h$'. Let $e(\Lambda)$ denote the number of edges that are labeled `$e$', or that are labeled `$h$' and have multiplicity greater than one.
\end{notation}

\begin{defn}
\cite[Def.\ 1.11]{Hutchings16}
If $\Lambda$ is a convex generator, define the {\em combinatorial ECH index\/} of $\Lambda$ by
\begin{equation}
\label{eqn:Ihat}
\widehat{I}(\Lambda) = 2\left(\widehat{\mathcal{L}}(\Lambda)-1\right) - h(\Lambda).
\end{equation}
Here $\widehat{\mathcal{L}}(\Lambda)$ denotes the number of lattice points in the polygon bounded by $\Lambda$, the line segment from $(0,0)$ to $(x(\Lambda),0)$, and the line segment from $(0,0)$ to $(0,y(\Lambda))$, including lattice points on the boundary. Also, define the {\em combinatorial $J_0$ index\/} of $\Lambda$ by
\begin{equation}
\label{eqn:Jhat}
\widehat{J}_0(\Lambda) = \widehat{I}(\Lambda) - 2x(\Lambda) - 2y(\Lambda) - e(\Lambda).
\end{equation}
\end{defn}

\begin{lem}
\label{lem:convexbijection}
(cf.\ \cite[Lem.\ 5.4]{Hutchings16})
Let $X_\Omega\subset\R^4$ be a convex toric domain, and let $L > \max(a(\Omega),b(\Omega))$. Then a perturbation $X$ of $X_\Omega$ as in Lemma~\ref{lem:Lniceconvex} can be chosen so that there is a bijection
\begin{equation}
\label{eqn:convexbijection}
\imath:
\left\{\begin{array}{cc}
\mbox{convex generators $\Lambda$}\\
\mbox{with $\mathcal{A}_\Omega(\Lambda)<L$}
\end{array}\right\}
\longrightarrow
\left\{\begin{array}{cc}
\mbox{ECH generators $\alpha$ in $\partial X$}\\
\mbox{with $\mathcal{A}(\alpha)<L$} 
\end{array}
\right\}
\end{equation}
such that if $\imath(\Lambda)=\alpha$, then
\begin{align}
\label{eqn:convexbijectionaction}
|\mathcal{A}(\alpha) - \mathcal{A}_\Omega(\Lambda)| &< L^{-1},\\
\label{eqn:convexbijectionindex}
I(\alpha) &= \widehat{I}(\Lambda),\\
\label{eqn:convexbijectionj0}
J_0(\alpha) &= \widehat{J}_0(\Lambda).
\end{align}
\end{lem}

\begin{proof}
The bijection $\imath$ is defined as follows. If $\Lambda$ is a convex generator, then $\imath(\Lambda)$ is the product over the edges of $\Lambda$ of the following contributions. Let $v$ be an edge of $\Lambda$ and write $v^\perp=(ma,mb)$ where $a,b\ge 0$ are relatively prime and $m$ is the multiplicity of $v$. If $v$ is labeled `$e$', then the contribution is $e_{a,b}^m$. If $v$ is labeled `$h$', then the contribution is $e_{a,b}^{m-1}h_{a,b}$. It follows from \eqref{eqn:Lniceconvexaction} that, possibly after choosing inputting a larger value of $L$ to Lemma~\ref{lem:Lniceconvex}, $\imath$ is a well-defined bijection \eqref{eqn:convexbijection} satisfying \eqref{eqn:convexbijectionaction}. The formulas \eqref{eqn:convexbijectionindex} and \eqref{eqn:convexbijectionj0} for $I$ and $J_0$ follow from equations \eqref{eqn:Lniceconvexlinking}--\eqref{eqn:ctauhab} as in \cite[\S5.3, Step 4]{Hutchings16}.
\end{proof}

\begin{remark}
\label{rem:he}
Under the bijection \eqref{eqn:convexbijection}, the total number of simple Reeb orbits that appear in $\alpha$ equals $e(\Lambda)+h(\Lambda)$.
\end{remark}

\subsubsection{Concave toric domains}

A variant of the above story holds for concave toric domains.

\begin{defn}
A {\em concave integral path\/} is a path $\Lambda$ in the plane such that:
\begin{itemize}
\item
The endpoints of $\Lambda$ are $(0,y(\Lambda))$ and $(x(\Lambda),0)$ where $x(\Lambda)$ and $y(\Lambda)$ are nonnegative integers.
\item
$\Lambda$ is the graph of a piecewise linear convex function $f:[0,x(\Lambda)]\to[0,y(\Lambda)]$ with $f'(0)< 0$ and $f(x(\Lambda))=0$.
\item
The vertices of $\Lambda$ are lattice points.
\end{itemize}
\end{defn}

\begin{defn}
If $X_\Omega$ is a concave toric domain and $\Lambda$ is a concave integral path, define the {\em $\Omega$-action\/} of $\Lambda$ to be
\[
\mathcal{A}_\Omega(\Lambda) = \sum_{v\in\op{Edges}(\Lambda)}\big[v^\perp\big]_\Omega.
\]
\end{defn}

\begin{defn}
A {\em concave generator\/} is a concave integral path $\Lambda$, together with a labeling of each edge by `$e$' or `$h$' (we omit the labeling from the notation). We define $h(\Lambda)$ and $e(\Lambda)$ as before.
\end{defn}

\begin{defn}
If $\Lambda$ is a concave generator, define the {\em combinatorial ECH index\/} of $\Lambda$ by
\begin{equation}
\label{eqn:checkI}
\widecheck{I}(\Lambda) = 2\widecheck{\mathcal{L}}(\Lambda) + h(\Lambda).
\end{equation}
Here $\widecheck{\mathcal{L}}(\Lambda)$ denotes the number of lattice points in the polygon bounded by $\Lambda$, the line segment from $(0,0)$ to $(x(\Lambda),0)$, and the line segment from $(0,0)$ to $(0,y(\Lambda))$, including lattice points on the boundary, except not including lattice points on $\Lambda$ itself. Also, define the {\em combinatorial $J_0$ index\/} of $\Lambda$ by
\begin{equation}
\label{eqn:Jcheck}
\widecheck{J}_0(\Lambda) = \widecheck{I}(\Lambda) - 2x(\Lambda) - 2y(\Lambda) + e(\Lambda).
\end{equation}
\end{defn}

The following is a special case of \cite[Lem.\ 3.3]{concave}.

\begin{lem}
\label{lem:concavebijection}
Let $X_\Omega\subset\R^4$ be a concave toric domain, and let $L > \max(a(\Omega),b(\Omega))$. Then a perturbation $X$ of $X_\Omega$ as in Lemma~\ref{lem:Lniceconcave} can be chosen so that there is a bijection
\[
\imath:
\left\{\begin{array}{cc}
\mbox{concave generators $\Lambda$ with}\\
\mbox{$\mathcal{A}_\Omega(\Lambda)<L$ and $\widecheck{I}(\Lambda) < L$}
\end{array}\right\}
\longrightarrow
\left\{\begin{array}{cc}
\mbox{ECH generators $\alpha$ in $\partial X$ with}\\
\mbox{$\mathcal{A}(\alpha)<L$ and $I(\alpha) < L$}
\end{array}
\right\}
\]
such that if $\imath(\Lambda) = \alpha$, then
\begin{align}
\nonumber
|\mathcal{A}(\alpha) - \mathcal{A}_\Omega(\Lambda)| &< L^{-1},\\
\nonumber
I(\alpha) &= \widecheck{I}(\Lambda),\\
\label{eqn:concavebijectionj0}
J_0(\alpha) &= \widecheck{J}_0(\Lambda).
\end{align}
\end{lem}

\begin{proof}
This follows from Lemma~\ref{lem:Lniceconcave}, similarly to the proof of Lemma~\ref{lem:convexbijection}.
\end{proof}

\begin{defn}
If $X_\Omega\subset\R^4$ is a convex or concave toric domain, we say that a star-shaped domain $X$ provided by Lemma~\ref{lem:convexbijection} or Lemma~\ref{lem:concavebijection} respectively is an {\em $L$-nice perturbation\/} of $X_\Omega$.
\end{defn}

\begin{remark}
\label{rem:combdiff}
There is also a combinatorial formula for the ECH differential $\partial_J$ for suitable $J$ on the ECH generators as described above for $L$-nice perturbations of convex and concave toric domains, similar to the differential for the ECH of $T^3$ \cite{HutchingsSullivan06} or the PFH of a Dehn twist \cite{HutchingsSullivan05} respectively. In principle this is proved in \cite{Choi}, although certain details are not fully explained. The formula in the convex case is stated in \cite[Conj.\ A.3]{Hutchings16}, and more details in the concave case are provided in \cite{Trejos}. Morse-Bott theory needed for this is worked out in \cite{yao1,yao2}.
\end{remark}


\subsection{Cobordism maps on ECH}
\label{ss:cob}

We now review cobordism maps on embedded contact homology and some of their properties in the special case that we need.

\begin{defn}
Let $(Y_+,\lambda_+)$ and $(Y_-,\lambda_-)$ be contact three-manifolds. A {\em strong symplectic cobordism\/} from\footnote{This usage of the words ``from'' and ``to'' is natural from the perspective of symplectic geometry, but opposite from most topology literature.} $(Y_+,\lambda_+)$ to $(Y_-,\lambda_-)$ is a compact symplectic four-manifold $(W,\omega)$ such that $\partial W = Y_+ - Y_-$ and $\omega|_{Y_\pm}=d\lambda_\pm$.
\end{defn}

Given a cobordism as above, one can find a neighborhood $N_-$ of $Y_-$ in $W$, identified with $[0,\epsilon)\times Y_-$ for some $\epsilon>0$, in which $\omega=e^s\lambda_-$, where $s$ denotes the $[0,\epsilon)$ coordinate. Likewise, one can choose a neighborhood $N_+\simeq (-\epsilon,0]\times Y_+$ of $Y_+$ in $W$ in which $\omega=e^s\lambda_+$. Fix a choice of neighborhoods $N_-$ and $N_+$. Using the neighborhood identifications, we can glue to form the {\em symplectic completion\/}
\[
\overline{W} = ((-\infty,0] \times Y_-) \cup_{Y_-} W \cup_{Y_+} ([0,\infty)\times Y_+).
\]

\begin{defn}
\label{def:cobordismadmissible}
An almost complex structure $J$ on $\overline{W}$ is {\em cobordism-admissible\/} if:
\begin{itemize}
\item
On $W$, the almost complex structure $J$ is $\omega$-compatible.
\item
On $(-\infty,0]\times Y_-$ and $[0,\infty)\times Y_+$, the almost complex structure $J$ agrees with the restrictions of $\lambda_\pm$-compatible almost complex structures $J_\pm$ on $\R\times Y_\pm$.
\end{itemize}
\end{defn}

Assume now that the contact forms $\lambda_\pm$ are nondegenerate. For a cobordism-admissible almost complex structure $J$ as above, if $\alpha_+$ is an orbit set for $\lambda_+$ and $\alpha_-$ is an orbit set for $\lambda_-$, then we define a moduli space $\mathcal{M}^J(\alpha_+,\alpha_-)$ of $J$-holomorphic currents in $\overline{W}$ analogously to the symplectization case in \S\ref{ss:ECHrev}.

More generally, we define a {\em broken $J$-holomorphic current\/} from $\alpha_+$ to $\alpha_-$ to be a tuple $(\mathcal{C}_{N_-},\ldots,\mathcal{C}_{N_+})$ where $N_-\le 0 \le N_+$, for which there exist orbit sets $\alpha_-=\alpha_-(N_-),\ldots,\alpha_-(0)$ in $Y_-$ and orbit sets $\alpha_+(0),\ldots,\alpha_+(N_+)=\alpha_+$ in $Y_+$, such that:
\begin{itemize}
\item
$\mathcal{C}_i\in\mathcal{M}^{J_-}(\alpha_-(i+1),\alpha_-(i))/\R$ for $i=N_-,\ldots,-1$.
\item
$\mathcal{C}_0\in\mathcal{M}^J(\alpha_+(0),\alpha_-(0))$.
\item
$\mathcal{C}_i\in\mathcal{M}^{J_+}(\alpha_+(i),\alpha_+(i-1))/\R$ for $i=1,\ldots,N_+$.
\item
If $i\neq 0$, then $\mathcal{C}_i$ is not $\R$-invariant.
\end{itemize}
We denote the set of such broken $J$-holomorphic currents by $\overline{\mathcal{M}^J}(\alpha_+,\alpha_-)$.

\begin{prop}
\label{prop:cobordism}
(special case of \cite[Thm.\ 3.5]{Hutchings16})
Let $(W,\omega)$ be a strong symplectic cobordism from $(Y_+,\lambda_+)$ to $(Y_-,\lambda_-)$ and assume that the contact forms $\lambda_\pm$ are nondegenerate. Assume also\footnote{The homological assumptions \eqref{eqn:cobha} can be dropped, if one restricts to the subspace of ECH generated by nullhomologous ECH generators and assumes that the cobordism $(W,\omega)$ is ``weakly exact''; see \cite[\S3.10]{Hutchings16}. In this case the sense in which the cobordism map respects the grading needs to be stated more carefully.} that
\begin{equation}
\label{eqn:cobha}
H_1(Y_\pm)=H_2(Y_\pm)=H_2(W)=0.
\end{equation}
Then for each $L\in\R$ there is a well-defined cobordism map
\begin{equation}
\label{eqn:PhiL}
\Phi^L(W,\omega) : ECH^L_*(Y_+,\lambda_+) \longrightarrow ECH^L_*(Y_-,\lambda_-)
\end{equation}
with the following properties:
\begin{description}
\item{(a)}
If $L<L'$, then the diagram
\[
\begin{CD}
ECH^L_*(Y_+,\lambda_+) @>{\Phi^L(W,\omega)}>> ECH^L_*(Y_-,\lambda_-)\\
@VVV @VVV\\
ECH^{L'}_*(Y_+,\lambda_+) @>{\Phi^{L'}(W,\omega)}>> ECH^{L'}_*(Y_-,\lambda_-)
\end{CD}
\]
commutes. In particular, we have a well-defined direct limit
\begin{equation}
\label{eqn:directlimit}
\Phi(W,\omega) = \lim_{L\to\infty} \Phi^L(X,\omega) : ECH_*(Y_+,\lambda_+) \longrightarrow ECH_*(Y_-,\lambda_-).
\end{equation}
\item{(b)}
If $W$ is diffeomorphic to a product $[0,1]\times Y$, then the map \eqref{eqn:directlimit} is an isomorphism.
\item{(c)}
Let $J_\pm$ be generic $\lambda_\pm$-compatible almost complex structures on $\R\times Y_\pm$, and let $J$ be any cobordism-admissible almost complex structure on $\overline{W}$ extending $J_\pm$. Then for each $L$, the cobordism map \eqref{eqn:PhiL} is induced by a (noncanonical) chain map
\[
\phi: (ECC^L(Y_+,\lambda_+),\partial_{J_+}) \longrightarrow (ECC^L(Y_-,\lambda_-),\partial_{J_-})
\]
with the following property: If $\alpha_\pm$ are ECH generators in $Y_\pm$, and if the coefficient $\langle\phi\alpha_+,\alpha_-\rangle\neq 0$, then there exists a broken $J$-holomorphic current $(\mathcal{C}_{N_-},\ldots,\mathcal{C}_{N_+}) \in \overline{\mathcal{M}^J}(\alpha_+,\alpha_-)$.
\end{description}
\end{prop}

\begin{remark}
\label{rem:action}
The homological assumptions \eqref{eqn:cobha} imply that in part (c), if $\langle\phi\alpha_+,\alpha_-\rangle \neq0$, or more generally if $\overline{\mathcal{M}^J}(\alpha_+,\alpha_-)\neq\emptyset$, then $\mathcal{A}(\alpha_+)\ge \mathcal{A}(\alpha_-)$, with equality only if $\alpha_+=\emptyset$. 
\end{remark}


\subsection{Special properties of ECH cobordism maps for toric domains}
\label{sec:cobordismspecial}

The construction in \cite{HutchingsTaubes13} of the cobordism map \eqref{eqn:PhiL} does not directly count holomorphic currents, due to difficulties with multiple covers, but rather uses Seiberg-Witten theory; see \cite[\S5.5]{Hutchingslec}. This is why in part (c), we only obtain a broken holomorphic current. However in some special cases, namely for ``$L$-tame'' cobordisms defined in \cite[\S4.1]{Hutchings16}, we obtain actual holomorphic currents.

In particular, suppose that $X_{\Omega_-}, X_{\Omega_+} \subset \R^4$ are convex or concave toric domains. Suppose further that $X_{\Omega_-}$ is a convex toric domain or $X_{\Omega_+}$ is a concave toric domain (i.e.\ we are not in the case where $X_{\Omega_-}$ is a concave toric domain and $X_{\Omega_+}$ is a convex toric domain, which is studied in \cite{Cristofaro-Gardiner19}). Suppose there exists a symplectic embedding $\varphi:X_{\Omega_-}\to\op{int}(X_{\Omega_+})$. By Lemma~\ref{lem:convexbijection} and/or Lemma~\ref{lem:concavebijection}, we can find $L$-nice approximations $X_-\subset X_{\Omega_-}$ and $X_+\supset X_{\Omega_+}$. Let $W=X_+\setminus\varphi(\op{int}(X_-))$; this is a symplectic cobordism from $(\partial X_+,\lambda_+)$ to $(\partial X_-,\lambda_-)$, where $\lambda_\pm$ is the restriction to $\partial X_\pm$ of the standard Liouville form \eqref{eqn:lambda0}.

\begin{lem}
\label{lem:tamecobordism}
In the above situation, if $J$ is a generic cobordism-admissible almost complex structure on $\overline{W}$, then:
\begin{description}
\item{(a)}
If $\mathcal{A}(\alpha_+)<L$ and $\mathcal{C}\in\mathcal{M}^J(\alpha_+,\alpha_-)$ is a $J$-holomorphic current, then $I(\alpha_+)\ge I(\alpha_-)$.
\item{(b)}
In Proposition~\ref{prop:cobordism}(c), the broken $J$-holomorphic current $(\mathcal{C}_{N_-},\ldots,\mathcal{C}_{N_+})\in\overline{\mathcal{M}^J}(\alpha_+,\alpha_-)$ satisfies $N_-=N_+=0$, so that we have a $J$-holomorphic current $\mathcal{C}_0\in\mathcal{M}^J(\alpha_+,\alpha_-)$.
\end{description}
\end{lem}

\begin{proof}
If $X_{\Omega_-}$ and $X_{\Omega_+}$ are both convex toric domains, then $W$ is an ``$L$-tame'' cobordism in the sense of \cite[Def.\ 4.3]{Hutchings16}, as shown in \cite[\S6]{Hutchings16}. The same is true when $X_{\Omega_+}$ is a concave toric domain, by a similar argument. Assertion (a) now follows from \cite[Prop.\ 4.6(a)]{Hutchings16}. Assertion (b) follows from (a) together with the fact that every non-$\R$-invariant $J_\pm$-holomorphic current in $\R\times Y_\pm$ has positive ECH index, as reviewed in \cite[Prop.\ 3.7]{Hutchingslec}.
\end{proof}

We will also need the following lemma regarding linking numbers.

\begin{lem}
\label{lem:convexlinking}
Let $X_{\Omega_+}$ and $X_{\Omega_-}$ be convex toric domains, and let $\varphi:X_{\Omega_-} \to \op{int}(X_{\Omega_+})$ be a symplectic embedding. Let $L, X_+, X_-, W$ be as above. Let $J$ be any cobordism-admissible almost complex structure on $\overline{W}$. Let $\alpha_+$ and $\alpha_-$ be convex generators with $\mathcal{A}(\alpha_+)<L$. Suppose there exists a holomorphic current $\mathcal{C}\in\mathcal{M}^J(\alpha_+,\alpha_-)$.
\begin{description}
\item{(a)}
If there exists a $J$-holomorphic curve $C_1\in\mathcal{M}^J(e_{1,0},e_{1,0})$, then $x(\alpha_+) \ge x(\alpha_-)$.
\item{(b)}
If there exists a $J$-holomorphic curve $C_2\in\mathcal{M}^J(e_{0,1},e_{0,1})$, then $y(\alpha_+) \ge y(\alpha_-)$.
\end{description}
\end{lem}

\begin{proof}
We follow the proof of \cite[Lem.\ 5.1]{calabi} with minor modifications.

We first prove assertion (a). We can assume without loss of generality that $\mathcal{C}$ consists of a single somewhere injective curve $C$ which is distinct from $C_1$. Let $s_+>>0$ and let
\[
\eta_+ = C\cap (\{s_+\}\times\partial X_+).
\]
By standard results on asymptotics of holomorphic curves, see e.g.\ \cite[Cor.\ 2.5, 2.6]{siefring}, if $s_+$ is sufficiently large then $\eta_+$ is cut out transversely and disjoint from the Reeb orbits in $\alpha_+$. Likewise, let $s_-<<0$ and let $\eta_- = C \cap (\{s_-\}\times \partial X_-)$; if $|s_-|$ is sufficiently large then $\eta_-$ is cut out transversely and disjoint from the Reeb orbits in $\alpha_-$.

Now observe that
\begin{equation}
\label{eqn:elleta}
\ell(\eta_+,e_{1,0}) - \ell(\eta_-,e_{1,0}) = \#(C\cap C_1) \ge 0.
\end{equation}
Here `$\#$' denotes the algebraic intersection number. The equality on the left holds by the definition of linking number, and the inequality on the right holds by intersection positivity for $J$-holomorphic curves.

To start to analyze the left hand side of \eqref{eqn:elleta}, we can write
\[
\eta_+ = \coprod_{(\gamma,m)\in\alpha_+}\eta^{+}_{\gamma}
\]
where $\eta^{+}_{\gamma}$ is a link in a tubular neighborhood of $\gamma$ which, in this tubular neighborhood, is homologous to $m\gamma$. By the definition of linking number, we have
\[
\ell(\eta_+,e_{1,0}) = \sum_{(\gamma,m)\in\alpha_+}\ell(\eta^{+}_{\gamma},e_{1,0}).
\]
If $(a,b)\neq (1,0)$ and $\gamma=e_{a,b}$ or $\gamma=h_{a,b}$, then it follows from equation \eqref{eqn:Lniceconvexlinking} that
\[
\ell(\eta^{+}_{\gamma},e_{1,0}) = mb.
\]
If $\gamma=e_{1,0}$, then it follows from the winding number bounds from \cite[\S3]{hwz}, which are reviewed in our notation in \cite[Lem.\ 5.3(b)]{Hutchingslec}, that
\[
\ell(\eta^{+}_{e_{1,0}},e_{1,0}) \le 0.
\]
Combining the above three lines, we conclude that
\begin{equation}
\label{eqn:elleta+}
\ell(\eta_+,e_{1,0}) \le x(\alpha_+).
\end{equation}

A similar calculation shows that
\begin{equation}
\label{eqn:elleta-}
\ell(\eta_-,e_{1,0}) \ge x(\alpha_-).
\end{equation}
(In fact, if $e_{1,0}$ appears in $\alpha_-$, then the inequality \eqref{eqn:elleta-} is strict.) Combining \eqref{eqn:elleta}, \eqref{eqn:elleta+}, and \eqref{eqn:elleta-} completes the proof of (a). Assertion (b) is proved by a symmetric argument.
\end{proof}

%% file: proofs.tex
\section{Proofs of the main theorems}
\label{sec:proofs}

We now prove the main results. Theorems~\ref{thm:polydiskball}, \ref{thm:convex1}, \ref{thm:convex2}, and \ref{thm:noanchor} are proved in \S\ref{sec:convexproofs}, and Theorem~\ref{thm:concave2} is proved in \S\ref{sec:concaveproofs}.


\subsection{Preliminary lemmas}

We begin with a lemma concerning the following geometric setup. Let $X_{\Omega_-}$ and $X_{\Omega_+}$ be convex toric domains, and suppose there exists a symplectic embedding
\[
\varphi:X_{\Omega_-}\longrightarrow\op{int}(X_{\Omega_+}).
\]
Let
\[
L>\max(a(\Omega_-),a(\Omega_+),b(\Omega_-),b(\Omega_+))
\]
and let $X_-\subset X_{\Omega_-}$ and $X_+\supset X_{\Omega_+}$ be $L$-nice approximations provided by Lemma~\ref{lem:convexbijection}. Write $Y_\pm=\partial X_\pm$ and let $\lambda_\pm$ denote the induced contact form on $Y_\pm$. By Lemma~\ref{lem:convexbijection}, ECH generators in $Y_-$ or $Y_+$ with symplectic action less than $L$ can be identified with convex generators with $\Omega_-$-action or $\Omega_+$-action less than $L$, respectively, via the bijection $\imath$, and we omit $\imath$ from the notation. Let $W$ be the symplectic cobordism from $(Y_+,\lambda_+)$ to $(Y_-,\lambda_-)$ given by $X_+\setminus\varphi(\op{int}(X_-))$ as in \S\ref{ss:cob}. Let $J$ be a cobordism-admissible almost complex structure on $\overline{W}$ as in Definition~\ref{def:cobordismadmissible}.

\begin{lem}
\label{lem:beyond}
Let $a,b\ge 0$ be relatively prime nonnegative integers, not both zero, and assume that $L>\mathcal{A}_{\Omega_+}(e_{a,b})$. Suppose there exist $J$-holomorphic curves $C_1\in\mathcal{M}^J(e_{1,0},e_{1,0})$ and $C_2\in\mathcal{M}^J(e_{0,1},e_{0,1})$. Let $\Lambda$ be an ECH generator in $Y_-$ with $\widehat{I}(\Lambda)=\widehat{I}(e_{a,b})$. Suppose there exists a $J$-holomorphic current $\mathcal{C}\in\mathcal{M}^J(e_{a,b},\Lambda)$. Then $\Lambda = e_{a,b}$.
\end{lem}

\begin{proof}
Since $e_{a,b}$ is a simple Reeb orbit, the current $\mathcal{C}$ consists of a single somewhere injective curve $C$ with multiplicity one\footnote{Since the symplectic form on $W$ is exact, every nonconstant $J$-holomorphic curve in $\overline{W}$ must have at least one positive end, as in Remark~\ref{rem:action}.}. It then follows from the $J_0$ bound \eqref{eqn:J0bound}, equation \eqref{eqn:convexbijectionj0}, and Remark~\ref{rem:he} that
\begin{equation}
\label{eqn:useJ0bound}
\widehat{J}_0(e_{a,b}) - \widehat{J}_0(\Lambda) \ge 2 g(C) - 1 + e(\Lambda) + h(\Lambda).
\end{equation}
Since $\widehat{I}(\Lambda)=\widehat{I}(e_{a,b})$, it follows from equation \eqref{eqn:Jhat} that
\begin{equation}
\label{eqn:useJ0formula}
\widehat{J}_0(e_{a,b}) - \widehat{J}_0(\Lambda) = 2(x(\Lambda)-b + y(\Lambda) - a) + e(\Lambda) - 1. 
\end{equation}
By Lemma~\ref{lem:convexlinking}, we have
\begin{equation}
\label{eqn:m3}
x(\Lambda) \le b
\end{equation}
and
\begin{equation}
\label{eqn:m4}
y(\Lambda) \le a.
\end{equation}
Combining \eqref{eqn:useJ0bound}, \eqref{eqn:useJ0formula}, \eqref{eqn:m3}, and \eqref{eqn:m4}, we obtain
\[
2 g(C) + h(\Lambda) \le 0,
\]
with equality only if $x(\Lambda) = b$ and $y(\Lambda) = a$. We conclude that
\[
g(C)=0, \quad\quad h(\Lambda) = 0, \quad\quad x(\Lambda)=b, \quad\quad y(\Lambda)=a.
\]
Since $\widehat{I}(\Lambda)=\widehat{I}(e_{a,b})$, the index formulas~\eqref{eqn:Ihat} and \eqref{eqn:convexbijectionindex}  imply that
\[
\widehat{\mathcal{L}}(\Lambda) = \widehat{\mathcal{L}}(e_{a,b}).
\]  
Since the path underlying $\Lambda$ is convex and has the same endpoints as the line segment corresponding to $e_{a,b}$, it follows that $\Lambda=e_{a,b}$. (We also get that $C$ is a cylinder, although we do not need this.)
\end{proof}


\begin{prop}
\label{prop:eab}
Let $X_\Omega$ be a convex toric domain and let $a,b\ge 0$ be relatively prime nonnegative integers. Let $X$ be an $L$-nice perturbation of $X_\Omega$ where $L$ is large with respect to $a$, $b$, and $\Omega$. Let $Y=\partial X$ and let $\lambda$ denote the induced contact form on $Y$. Let $J_-$ be a generic $\lambda$-compatible almost complex structure on $\R\times Y$. Then $e_{a,b}$ is a cycle in $(ECC^L_*(Y,\lambda),\partial_{J_-})$ which represents a nonzero homology class in $ECH_*(Y,\lambda)$.
\end{prop}

\begin{proof}
We proceed in four steps.

{\em Step 1:} We first prove the proposition in the special case where $a,b>0$ and $X_\Omega$ is the ellipsoid $E(bc,ac)$ where $c>0$ is a positive real number.

By \cite[Lem. 2.1(a)]{Hutchings16}, $e_{a,b}$ is ``minimal'' for $E(bc,ac)$ in the sense of \cite[Def.\ 1.15]{Hutchings16}. Minimality means that $e_{a,b}$ uniquely minimizes $\Omega$-action among all convex generators with the same ECH index as $e_{a,b}$ and with all edges labeled `$e$'. The proposition in this case now follows from \cite[Lem.\ 5.5]{Hutchings16}.

{\em Step 2:} We now set up the proof of the proposition in the general case where $a,b>0$.

Let $c>0$ be a positive real number which is sufficiently large that $X_\Omega \subset \operatorname{int}(E(bc,ac))$. Let $X_-\subset X_\Omega$ and $X_+\supset E(bc,ac)$ be $L$-nice perturbations of $X_\Omega$ and $E(bc,ac)$ respectively, where $L$ is large with respect to $a$, $b$, $c$, and $\Omega$. Let $Y_\pm$ and $W$ be as in the statement of Lemma~\ref{lem:beyond}. Let $J$ be a cobordism-admissible almost complex structure on $\overline{W}$ which restricts to $J_-$ on $(-\infty,0]\times Y_-$, and let $J_+$ denote the restriction of $J$ to $[0,\infty)\times Y_+$.

Define a surface $C_1\subset\overline{W}$ to be the union of $W\cap(\C\times\{0\})$ with the ``trivial half-cylinders'' $(-\infty,0]\times e_{1,0}$ in $(-\infty,0]\times Y_-$ and $[0,\infty)\times e_{1,0}$ in $[0,\infty)\times Y_+$. Likewise, define $C_2\subset\overline{W}$ to be the union of $W\cap(\{0\}\times \C)$ with the trivial half-cylinders over $e_{0,1}$ in $(-\infty,0]\times Y_-$ and $[0,\infty)\times Y_+$. By the definition of $\lambda_\pm$-compatible almost complex structure, $C_1$ and $C_2$ are necessarily $J$-holomorphic in $(-\infty,0]\times Y_-$ and $[0,\infty)\times Y_+$. We can choose $J$ so that $C_1$ and $C_2$ are $J$-holomorphic in $W$ as well.

Next, perturb $J$ to be generic as needed for Lemma~\ref{lem:tamecobordism}. By standard automatic transversality arguments, the holomorphic cylinders $C_1$ and $C_2$ persist under a sufficiently small perturbation. (See \cite{wendl} for a general treatment of automatic transversality, and \cite[Lem.\ 4.1]{hn} for a simple special case which is sufficient for the present situation.) We use the same notation $J$, $C_1$, and $C_2$ for the new almost complex structure and holomorphic cylinders.

{\em Step 3.\/} We now complete the proof of the proposition in the general case where $a,b>0$.

Let
\[
\phi: (ECC_*^L(Y_+,\lambda_+,),J_+) \longrightarrow (ECC_*^L(Y_-,\lambda_-),J_-)
\]
be a chain map as provided by Proposition~\ref{prop:cobordism}(c). Since $e_{a,b}$ is a cycle representing a nontrivial homology class in $ECH(Y_+,\lambda_+)$ by Step 1, it follows from Proposition~\ref{prop:cobordism}(a),(b) that
\[
\phi(e_{a,b}) \in ECC_*^L(Y_-,\lambda_-)
\]
is a cycle representing a nontrivial homology class in $ECH(Y_-,\lambda_-)$.

Let $\Lambda$ be an ECH generator for $(Y_-,\lambda_-)$, and suppose that $\langle\phi(e_{a,b}),\Lambda\rangle\neq 0$. By Lemma~\ref{lem:tamecobordism}(b), there exists a $J$-holomorphic current $\mathcal{C}\in\mathcal{M}^J(e_{a,b},\Lambda)$. By Lemma~\ref{lem:beyond}, we have $\Lambda=e_{a,b}$.

It follows from the previous paragraph that $\phi(e_{a,b})$ equals either $e_{a,b}$ or zero.  But we know that $\phi(e_{a,b})$ represents a nontrivial homology class in $ECH(Y_-,\lambda_-)$, so we must have $\phi(e_{a,b})=e_{a,b}$.

{\em Step 4.\/} It remains to prove that $e_{1,0}$ and $e_{0,1}$ are cycles which represent the nonzero homology class in $ECH_2(Y,\lambda)$.

We will restrict attention to $e_{1,0}$, as the proof for $e_{0,1}$ follows by a symmetric argument. As in Step 1, the claim holds if $X_\Omega = E(bc,c)$ where $b>1$ and $c>0$. The claim for general $X_\Omega$ then follows by repeating Steps 2 and 3.
\end{proof}

\begin{remark}
For suitable almost complex structures $J_-$, Proposition~\ref{prop:eab} would also follow from the combinatorial formula for the ECH differential\footnote{A combinatorial formula for the ECH differential on a different (not $L$-nice) perturbation of the ellipsoid $E(a,b)$ is computed in \cite{nw}.} in \cite[Conj.\ A.3]{Hutchings16}, together with algebraic calculations as in \cite[Prop.\ 5.9]{HutchingsSullivan06}.
\end{remark}

\begin{lem}
\label{lem:phiab}
Under the assumptions of Lemma~\ref{lem:beyond}, let
\[
\phi: (ECC(Y_+,\lambda_+),\partial_{J_+}) \longrightarrow (ECC(Y_-,\lambda_-),\partial_{J_-})
\]
be a chain map as provided by Proposition~\ref{prop:cobordism}(c). Then
\[
\phi(e_{a,b}) = e_{a,b}.
\]
\end{lem}

\begin{proof}
By Proposition~\ref{prop:eab}, $e_{a,b}$ is a cycle in the chain complex $(ECC(Y_+,\lambda_+,\partial_{J_+})$ representing a nontrivial element of $ECH(Y_+,\lambda_+)$. We now repeat Step 3 of the proof of Proposition~\ref{prop:eab} to conclude that $\phi(e_{a,b}) = e_{a,b}$.
\end{proof}


\subsection{Results for convex toric domains}
\label{sec:convexproofs}

\begin{proof}[Proof of Theorem \ref{thm:polydiskball}]
The ``if'' part of the theorem follows from Remark~\ref{rem:inclusion}, so we just need to prove the ``only if'' part. Assume that $a>1$ and that there exists an anchored symplectic embedding
\[
(\varphi,\Sigma) :(P(a,1),e_{1,0}) \longrightarrow (B^4(c),e_{1,0}).
\]
We need to show that $c> a+1$. By straightening $\Sigma$ near the boundary (see below), we can find an anchored symplectic embedding into $(B^4(c - \varepsilon),e_{1,0})$ for some $\varepsilon>0$. Thus it is enough to show that $c \ge a + 1$.

Fix $L>>a, c$. Let $X_-\subset P(a,1)$ and $X_+\supset B^4(c)$ be $L$-nice approximations, and let $W$ be the resulting cobordism from $(Y_+,\lambda_+)$ to $(Y_-,\lambda_-)$ as at the beginning of \S\ref{sec:convexproofs}.

Note that the anchor $\Sigma$ is necessarily transverse to $\partial P(a,1)$ and $\partial B^4(c)$. We can then perform a symplectic isotopy to straighten $\Sigma$ near the boundary. Using the fourth bullet point in Lemma~\ref{lem:Lniceconvex}, we can then construct an embedded symplectic surface $C_1$ in $\overline{W}$ such that
\[
\begin{split}
C_1\cap((-\infty,0]\times Y_-) &= (-\infty,0]\times e_{1,0},\\
C_1\cap([0,\infty)\times Y_+) &= [0,\infty)\times e_{1,0}.
\end{split}
\]

We can choose a cobordism-admissible almost complex structure $J$ on $\overline{W}$ such that $C_1$ is $J$-holomorphic. It follows from the relative adjunction formula in \cite[Eq.\ (3.3)]{Hutchingslec}, or similary from \eqref{eqn:J0bound}, that $C_1$ is a cylinder. Thus, as in the proof of Proposition~\ref{prop:eab}, the curve $C_1$ satisfies automatic transversality. Then if we make a sufficiently small perturbation of $J$ to be generic as needed for Lemma~\ref{lem:tamecobordism}, the curve $C_1$ will persist. Let
\[
\phi: (ECC(Y_+,\lambda_+),\partial_{J_+}) \longrightarrow (ECC(Y_-,\lambda_-),\partial_{J_-})
\]
be a chain map as provided by Proposition~\ref{prop:cobordism}(c) for this perturbed $J$.

By Proposition~\ref{prop:eab}, $e_{1,1}$ is a cycle in $(ECC(Y_+,\lambda_+),\partial_{J_+})$ which represents the nonzero class in $ECH_4(Y_+,\lambda_+)$. Then $\phi(e_{1,1})$ represents the nontrivial class in $ECH_4(Y_-,\lambda_-)$. In particular, $\phi(e_{1,1})\neq 0$. The only ECH generators in $Y_-$ with ECH index 4 are $e_{1,0}^2$, $e_{1,1}$, and $e_{0,2}^2$, so at least one of these must appear in $\phi(e_{1,1})$. 

We must have $\langle\phi e_{1,1},e_{0,1}^2\rangle = 0$ as in \eqref{eqn:m3}. Therefore $e_{1,0}^2$ or $e_{1,1}$ appears in $\phi(e_{1,1})$. 

By Remark~\ref{rem:action}, the chain map $\phi$ respects the symplectic action filtration. Recall from \eqref{eqn:convexbijectionaction} that the actions of convex generators are approximated by the $\Omega$-action in Definition~\ref{def:Omegaaction}. In particular, we compute that
\[
\mathcal{A}_{B^4(c)}(e_{1,1}) = c
\]
and
\[
\mathcal{A}_{P(a,1)}(e_{1,0}^2) = 2a, \quad\quad \mathcal{A}_{P(a,1)}(e_{1,1}) = a+1.
\]
Therefore, by \eqref{eqn:convexbijectionaction} and Remark~\ref{rem:action}, we have
\[
c + 2L^{-1} > \min(2a,a+1) = a+1.
\]
Since $L$ can be chosen arbitrarily large, we conclude that $c\ge a+1$.
\end{proof}

\begin{lem}
\label{lem:Aab}
Let $X_{\Omega_-}$ and $X_{\Omega_+}$ be convex toric domains in $\R^4$. Suppose that for every pair of positive relatively prime integers $a,b>0$, we have
\begin{equation}
\label{eqn:Aab}
\mathcal{A}_{\Omega_-}(e_{a,b}) \le \mathcal{A}_{\Omega_+}(e_{a,b}).
\end{equation}
Then $\Omega_- \subset \Omega_+$.
\end{lem}

\begin{proof}
If $X_\Omega$ is a convex toric domain and $a,b>0$ are relatively prime positive integers, let $\Omega^{a,b}\subset\R^2$ denote the closed half-space to the lower left of the tangent line to $\partial_+\Omega$ with slope $-a/b$. Then
\begin{equation}
\label{eqn:hs1}
\Omega = \R^2_{\ge 0} \cap \bigcap_{a,b}\Omega^{a,b}
\end{equation}
where the intersection is over pairs $(a,b)$ of relatively prime positive integers.
By Definition~\ref{def:Omegaaction}, the hypothesis \eqref{eqn:Aab} is equivalent to the statement that
\begin{equation}
\label{eqn:hs2}
\Omega_-^{a,b} \subset \Omega_+^{a,b}.
\end{equation}
It follows from \eqref{eqn:hs1} and \eqref{eqn:hs2} that $\Omega_-\subset\Omega_+$.
\end{proof}

\begin{proof}[Proof of Theorem~\ref{thm:convex2}.]
Suppose there exists a $2$-anchored symplectic embedding
\[
(\varphi,\Sigma_1,\Sigma_2): (X_{\Omega_-},e_{1,0},e_{0,1}) \longrightarrow (X_{\Omega_+},e_{1,0},e_{0,1}).
\]
We need to show that $\Omega_-\subset\Omega_+$. By Lemma~\ref{lem:Aab}, it is enough to show that if $a,b>0$ are relatively prime positive integers, then the action inequality \eqref{eqn:Aab} holds.

Fix $L>>a,b$. Let $X_\pm$, $Y_\pm$, and $W$ be as in Lemma~\ref{lem:beyond}. As in the proof of Theorem~\ref{thm:polydiskball}, from the anchors we can construct disjoint embedded symplectic cylinders $C_1, C_2$ in $\overline{W}$, such that
\[
\begin{split}
C_1\cap((-\infty,0]\times Y_-) &= (-\infty,0]\times e_{1,0},\\
C_1\cap([0,\infty)\times Y_+) &= [0,\infty)\times e_{1,0},\\
C_2\cap((-\infty,0]\times Y_-) &= (-\infty,0]\times e_{0,1},\\
C_2\cap([0,\infty)\times Y_+) &= [0,\infty)\times e_{0,1}.
\end{split}
\]
As in the proof of Proposition~\ref{prop:eab}, Step 3, we can choose a generic cobordism-admissible almost complex structure $J$ on $\overline{W}$ so that perturbations of $C_1$ and $C_2$, which we still denote by $C_1$ and $C_2$, are $J$-holomorphic. In particular, $C_1\in\mathcal{M}^J(e_{1,0},e_{1,0})$, and $C_2\in\mathcal{M}^J(e_{0,1},e_{0,1})$.

Now let
\[
\phi: ECC(Y_+,\lambda_+,J_+) \longrightarrow ECC(Y_-,\lambda_-,J_-)
\]
be a chain map as provided by Proposition~\ref{prop:cobordism}(c). By Lemma~\ref{lem:phiab}, we have
\[
\phi(e_{a,b}) = e_{a,b}.
\]
By \eqref{eqn:convexbijectionaction} and Remark~\ref{rem:action}, we have
\[
\mathcal{A}_{X_-}(e_{a,b}) < \mathcal{A}_{X_+}(e_{a,b}) + 2L^{-1}.
\]
Since $L$ can be chosen arbitrarily large, the desired inequality \eqref{eqn:Aab} follows.
\end{proof}

\begin{proof}[Proof of Theorem~\ref{thm:convex1}.]
This is a slight variation on the proof of Theorem~\ref{thm:convex2}.
Suppose there exists an anchored symplectic embedding
\[
(\varphi,\Sigma): (X_{\Omega_-},e_{1,0}) \longrightarrow (X_{\Omega_+},e_{1,0}).
\]
Suppose also that $a(\Omega_-)>b(\Omega_+)$, or equivalently
\begin{equation}
\label{eqn:Ae10}
\mathcal{A}_{\Omega_-}(e_{1,0}) > \mathcal{A}_{\Omega_+}(e_{0,1}).
\end{equation}
We need to show that $\Omega_-\subset\Omega_+$. By Lemma~\ref{lem:Aab}, it is enough to show that if $a,b>0$ are relatively prime positive integers, then the inequality \eqref{eqn:Aab} holds.

To prove \eqref{eqn:Aab}, let $L>>a,b$, and let $X_\pm$, $Y_\pm$, and $W$ be as in the statement of Lemma~\ref{lem:beyond}. From the anchor we can construct an embedded symplectic surface $C_1$ in $\overline{W}$ such that
\[
\begin{split}
C_1\cap((-\infty,0]\times Y_-) &= (-\infty,0]\times e_{1,0},\\
C_1\cap([0,\infty)\times Y_+) &= [0,\infty)\times e_{1,0}.
\end{split}
\]
We can choose a generic cobordism-admissible almost complex structure $J$ on $\overline{W}$ so that a perturbation of $C_1$, which we still denote by $C_1$, is $J$-holomorphic. In particular, $C_1\in\mathcal{M}^J(e_{1,0},e_{1,0})$. 

Now let
\[
\phi: ECC(Y_+,\lambda_+,J_+) \longrightarrow ECC(Y_-,\lambda_-,J_-)
\]
be a chain map as provided by Proposition~\ref{prop:cobordism}. We know from Proposition~\ref{prop:eab} that $e_{0,1}$ is a cycle in $(ECC(Y_+,\lambda)_,\partial_{J_+})$ which represents the nonzero class in $ECH_2(Y_+,\lambda_+)$. Then $\phi(e_{0,1})$ is a cycle in $(ECC(Y_-,\lambda_-),\partial_{J_-})$ which represents the nonzero class in $ECH_2(Y_-,\lambda_-)$. The only possibilities are either $\phi(e_{0,1})=e_{0,1}$ or $\phi(e_{0,1})=\phi(e_{1,0})$. If $L$ is chosen sufficiently large, then the latter possibility is ruled out by the action hypothesis \eqref{eqn:Ae10}, so $\phi(e_{0,1})=e_{0,1}$.

By Lemma~\ref{lem:tamecobordism}(b), there exists a $J$-holomorphic curve $C_2 \in \mathcal{M}^J(e_{0,1},e_{0,1})$. Given the $J$-holomorphic curves $C_1$ and $C_2$, we now complete the proof as in the last paragraph of the proof of Theorem~\ref{thm:convex2}.
\end{proof}

\begin{proof}[Proof of Theorem~\ref{thm:noanchor}.]
By Remark~\ref{rem:eta}, we just need to prove the ``only if'' part of the theorem.
Suppose there exists a symplectic surface $\Sigma\subset X_{\Omega'}$ with
\[
\partial\Sigma = e_{0,1}(X_{\Omega'}) - e_{1,0}(X_\Omega).
\]
We need to prove the inequality \eqref{eqn:bxy}. As in the proof of Theorem~\ref{thm:polydiskball}, it is sufficient to prove the non-strict version of this inequality. In the notation of Definition~\ref{def:Omegaaction}, this inequality is equivalent to
\begin{equation}
\label{eqn:bxyequiv}
\mathcal{A}_{\Omega'}(e_{0,1}) \ge \mathcal{A}_\Omega(h_{1,1}).
\end{equation}

Choose $L>\mathcal{A}_{\Omega'}(e_{0,1})$. By Lemma~\ref{lem:convexbijection}, we can choose $L$-nice approximations $X_-\subset X_\Omega$ and $X_+\subset X_{\Omega'}$ with $X_-\subset X_+$. Let $W$ be the symplectic cobordism at the beginning of \S\ref{sec:convexproofs}.

As in the proof of Theorem~\ref{thm:polydiskball}, from the surface $\Sigma$ we can construct a surface $C_1$ in $\overline{W}$, and we can choose a cobordism-admissible almost complex structure $J_1$ on $\overline{W}$, such that $C_1\in\mathcal{M}^{J_1}(e_{0,1},e_{1,0})$. By the relative adjunction formula (e.g.\ as packaged by the $J_0$ index), $C_1$ must be a cylinder. Then, as in the proof of Proposition~\ref{prop:eab}, the curve $C_1$ satisfies automatic transversality.

Likewise, as in the proof of Proposition~\ref{prop:eab}, from the surface $W\cap(\{0\}\times \mathbb{C})$ we can construct a surface $C_2\in\overline{W}$, and we can choose a cobordism-admissible almost complex structure $J_2$ on $\overline{W}$, such that $C_2\in\mathcal{M}^{J_2}(e_{0,1},e_{0,1})$. 

By Lemma~\ref{lem:convexlinking}(b), we have $\mathcal{M}^{J_2}(e_{0,1},e_{1,0}) = \emptyset$.

Let $\{J_\tau\}_{1\le\tau\le 2}$ be a generic one-parameter family of cobordism-admissible almost complex structures on $\overline{W}$ interpolating between $J_1$ and $J_2$ above. Since $\mathcal{M}^{J_1}(e_{0,1},e_{1,0})$ is nonempty, and since automatic transversality holds for the moduli spaces $\mathcal{M}^{J_\tau}(e_{0,1},e_{1,0})$, some breaking must occur for $\tau\in(1,2)$. That is, there exists $\tau\in(1,2)$ and a broken holomorphic current
\[
(\mathcal{C}_{N_-},\ldots,\mathcal{C}_{N_+}) \in \overline{\mathcal{M}^{J_\tau}}(e_{0,1},e_{1,0})
\]
as in \S\ref{ss:cob} with $N_+ > N_-$.

We claim that $N_+=0$. Suppose to the contrary that $N_+>0$. Let $J_\tau^+$ denote the almost complex structure on $\R\times\partial X_+$ determined by the restriction of $J_\tau$ to $[0,\infty)\times \partial X_+$. Then $\mathcal{C}_{N_+}$ is a somewhere-injective holomorphic curve in $\mathcal{M}^{J_\tau^+}(e_{0,1},\alpha)$, and it follows from the ECH index inequality in \cite[Thm.\ 4.15]{Hutchings09} (see also the exposition in \cite[\S3.4]{Hutchingslec}) that $I(\alpha)\le 2$. The only possibility is that $\alpha=e_{1,0}$. But such a holomorphic curve cannot exist by automatic transversality as in the proof of Proposition~\ref{prop:eab}, Step 2.

Since $N_+=0$, it follows that $\mathcal{C}_0$ is a somewhere injective curve in $\mathcal{M}^{J_\tau}(e_{0,1},\alpha)$, and by the ECH index inequality (loc. cit.) we have $I(\alpha)\le 3$. The only possibility is that $\alpha=h_{1,1}$. Then by \eqref{eqn:convexbijectionaction} and Remark~\ref{rem:action}, we have
\[
\mathcal{A}_{\Omega_+}(e_{0,1}) \ge \mathcal{A}_{\Omega_-}(h_{1,1}) - 2 L^{-1}.
\]
Since we can choose $L$ arbitrarily large, this implies the desired inequality \eqref{eqn:bxyequiv}.
\end{proof}


\subsection{Results for concave toric domains}
\label{sec:concaveproofs}

To finish up, we now prove Theorem~\ref{thm:concave2} on 2-anchored symplectic embeddings of concave toric domains. The proof is very similar, and in some sense ``dual'', to the proof of Theorem~\ref{thm:convex2} for convex toric domains.

Sinilarly to the beginning of \S\ref{sec:convexproofs}, let $X_{\Omega_-}$ and $X_{\Omega_+}$ be concave toric domains, and suppose there exists a symplectic embedding $\varphi:X_{\Omega_-}\to\op{int}(X_{\Omega_+})$. Let $L>\max(a(\Omega_-),a(\Omega_+),b(\Omega_-),b(\Omega_+))$ and let $X_-\subset X_{\Omega_-}$ and $X_+\supset X_{\Omega_+}$ be $L$-nice approximations provided by Lemma~\ref{lem:concavebijection}. Assume also that the Reeb orbits $e_{1,0}$ and $e_{0,1}$ in $X_-$ and $X_+$ all have the same (large) rotation number (we will need this for automatic transversality below). Write $Y_\pm=\partial X_\pm$ and let $\lambda_\pm$ denote the induced contact form on $Y_\pm$. By Lemma~\ref{lem:concavebijection}, ECH generators in $Y_-$ or $Y_+$ with symplectic action and ECH index less than $L$ can be identified with concave generators with combinatorial ECH index and $\Omega_-$-action or $\Omega_+$-action less than $L$, respectively, via the bijection $\imath$, and we omit $\imath$ from the notation. Let $W$ be the symplectic cobordism from $(Y_+,\lambda_+)$ to $(Y_-,\lambda_-)$ given by $X_+\setminus\varphi(\op{int}(X_-))$ as in \S\ref{ss:cob}. Let $J$ be a cobordism-admissible almost complex structure on $\overline{W}$ as in Definition~\ref{def:cobordismadmissible}.

\begin{lem}
\label{lem:beyondconcave}
Let $a,b>0$ be relatively prime positive integers, and let $\Lambda$ be a concave generator with $\widecheck{I}(\Lambda)=\widecheck{I}(e_{a,b})$. Assume that $L>\mathcal{A}_{\Omega_+}(\Lambda)$. Suppose there exist $J$-holomorphic cylinders $C_1\in\mathcal{M}^J(e_{1,0},e_{1,0})$ and $C_2\in\mathcal{M}^J(e_{0,1},e_{0,1})$ and a $J$-holomorphic current $\mathcal{C}\in\mathcal{M}^J(\Lambda,e_{a,b})$. Then $\Lambda = e_{a,b}$.
\end{lem}

\begin{proof}
To start, we claim that $\mathcal{C}$ consists of a single somewhere injective curve $C$ with multiplicity $1$. To prove this, let $C_0\in\mathcal{M}^J(\Lambda_0,e_{a,b})$ denote the component of $\mathcal{C}$ with a negative end asymptotic to $e_{a,b}$, and write $\mathcal{C}_1=\mathcal{C}-C_0\in\mathcal{M}^J(\Lambda_1,\emptyset)$. Thus $\Lambda=\Lambda_0\Lambda_1$. We need to show that $\mathcal{C}_1=0$. If not, then $\Lambda_1\neq\emptyset$. It then follows using the index formula \eqref{eqn:checkI} that $\widecheck{I}(\Lambda_0) < \widecheck{I}(\Lambda) = \widecheck{I}(e_{a,b})$. Then the existence of $C$ contradicts Lemma~\ref{lem:tamecobordism}(a).

It follows from the $J_0$ bound \eqref{eqn:J0bound}, equation \eqref{eqn:concavebijectionj0}, and Remark~\ref{rem:he} that
\begin{equation}
\label{eqn:useJ0boundconcave}
\widecheck{J}_0(\Lambda) - \widecheck{J}_0(e_{a,b}) \ge 2 g(C) - 1 + e(\Lambda) + h(\Lambda).
\end{equation}
Since $\widecheck{I}(\Lambda)=\widecheck{I}(e_{a,b})$, it follows from equation \eqref{eqn:Jcheck} that
\begin{equation}
\label{eqn:useJ0formulaconcave}
\widecheck{J}_0(\Lambda) - \widecheck{J}_0(e_{a,b}) = 2(b-x(\Lambda) + a-y(\Lambda)) + e(\Lambda) - 1. 
\end{equation}
Since $C$ has positive intersections with $C_2$, as in Lemma~\ref{lem:convexlinking} we have
\begin{equation}
\label{eqn:m3concave}
x(\Lambda) \ge b.
\end{equation}
(This is simpler than the convex case in Lemma~\ref{lem:convexlinking} because $C$ does not have any ends asymptotic to $e_{0,1}$.)
Likewise, since $C$ has positive intersections with $C_1$, we have
\begin{equation}
\label{eqn:m4concave}
y(\Lambda) \ge a.
\end{equation}

Combining \eqref{eqn:useJ0boundconcave}, \eqref{eqn:useJ0formulaconcave}, \eqref{eqn:m3concave}, and \eqref{eqn:m4concave}, we obtain
\[
2 g(C) + h(\Lambda) \le 0,
\]
with equality only if $x(\Lambda) = b$ and $y(\Lambda) = a$. The rest of the argument proceeds as in the proof of Lemma~\ref{lem:beyond}.
\end{proof}

\begin{prop}
\label{prop:eabconcave}
Let $X_\Omega$ be a concave toric domain and let $a,b> 0$ be positive integers. Let $X$ be an $L$-nice perturbation of $X_\Omega$ where $L$ is large with respect to $a$, $b$, and $\Omega$. Let $Y=\partial X$ and let $\lambda$ denote the induced contact form on $Y$. Let $J_+$ be a generic $\lambda$-compatible almost complex structure on $\R\times Y$. Let $x$ be a cycle in $(ECC^L(Y,\lambda),J_+)$ representing the nonzero class in $ECH_*(Y,\lambda)$ with grading $\widecheck{I}(e_{a,b})$. Then $e_{a,b}$ is a summand in $x$.
\end{prop}

\begin{proof}
We first prove the proposition when $X_\Omega$ is an ellipsoid $E(bc,ac)$ for $c>0$. In this case, similarly to \cite[Lem.\ 2.1(a)]{Hutchings16}, $e_{a,b}$ is ``maximal'' in the sense that it uniquely maximizes $\Omega$-action among all concave generators with the same ECH index as $e_{a,b}$ and with all edges labeled `$e$'. The proposition then follows similarly\footnote{The proof of \cite[Lem.\ 5.5]{Hutchings16}, which applies to convex toric domains, uses the formula for the ECH capacities of convex toric domains in \cite[Lem.\ 5.6]{Hutchings16}. In the present situation we instead need to use the formula for the ECH capacities of concave toric domains in \cite[Thm.\ 1.21]{concave}.} to \cite[Lem.\ 5.5]{Hutchings16}.

To prove the proposition for a general concave toric domain $X_\Omega$, choose $c>0$ sufficiently small so that $E(bc,ac)\subset\op{int}(X_\Omega)$. Let $X_-\subset E(bc,ac)$ and $X_+\supset X_\Omega$ be $L$-nice perturbations, where $L$ is large with respect to $a$, $b$, and $\Omega$. Let $Y_\pm$ and $W$ be as in the statement of Lemma~\ref{lem:beyondconcave}. Let $J$ be a cobordism-admissible almost complex structure on $\overline{W}$ which restricts to $J_+$ on $[0,\infty)\times Y_+$, and let $J_-$ denote the restriction of $J$ to $(\infty,0] \times Y_-$.

As in Step 2 of the proof of Proposition~\ref{prop:eab}, we can choose $J$ to be generic as needed for Lemma~\ref{lem:tamecobordism} and so that there exist $J$-holomorphic cylinders $C_1\in\mathcal{M}^J(e_{1,0},e_{1,0})$ and $C_2\in\mathcal{M}^J(e_{0,1},e_{0,1})$.

Let
\[
\phi: (ECC_*^L(Y_+,\lambda_+,),J_+) \longrightarrow (ECC_*^L(Y_-,\lambda_-),J_-)
\]
be a chain map as provided by Proposition~\ref{prop:cobordism}(c). By Proposition~\ref{prop:cobordism}(a),(b), $\phi(x)$ is a cycle representing the nonzero generator in $ECH(Y_-,\lambda_-)$ with grading $\widecheck{I}(e_{a,b})$. By the ellipsoid case, $e_{a,b}$ is a summand in $\phi(x)$. Therefore $x$ contains a summand $\Lambda$ with $\langle \phi(\Lambda),e_{a,b}\rangle\neq 0$. By Lemma~\ref{lem:tamecobordism}(b), there exists a $J$-holomorphic current $\mathcal{C}\in\mathcal{M}^J(\Lambda,e_{a,b})$. By Lemma~\ref{lem:beyondconcave}, $\Lambda=e_{a,b}$.
\end{proof}

\begin{remark}
For suitable $J_+$, Proposition~\ref{prop:eabconcave} can also be deduced from the formula for the ECH differential in \cite[Prop.\ 3.3]{Trejos}, together with some algebraic calculations similar to \cite{HutchingsSullivan06}.
\end{remark}

\begin{lem}
\label{lem:phiabconcave}
Under the assumptions of Lemma~\ref{lem:beyondconcave}, let
\[
\phi: (ECC(Y_+,\lambda_+),\partial_{J_+}) \longrightarrow (ECC(Y_-,\lambda_-),\partial_{J_-})
\]
be a chain map as provided by Proposition~\ref{prop:cobordism}(c). Then
\[
\langle \phi(e_{a,b}),e_{a,b}\rangle \neq 0.
\]
\end{lem}

\begin{proof}
Let $x$ be a cycle in $(ECC(Y_+,\lambda_+),\partial_{J_+})$ representing the nonzero homology class of grading $\widecheck{I}(e_{a,b})$.  It follows from Proposition~\ref{prop:cobordism}(a),(b) that $\phi(x)$ is a cycle representing the nonzero class in $ECH(Y_-,\lambda_-)$ of the same grading. We know from Proposition~\ref{prop:eabconcave} that $e_{a,b}$ is a summand in $\phi(x)$. Therefore there is a summand $\Lambda$ in $x$ with $\langle\phi(\Lambda),e_{a,b}\rangle\neq 0$. It follows from Lemma~\ref{lem:tamecobordism}(b) and Lemma~\ref{lem:beyondconcave} that $\Lambda=e_{a,b}$.
\end{proof}

\begin{lem}
\label{lem:Aabconcave}
Let $X_{\Omega_-}$ and $X_{\Omega_+}$ be concave toric domains in $\R^4$. Suppose that for every pair of positive relatively prime integers $a,b>0$, we have
\begin{equation}
\label{eqn:Aabconcave}
A_{\Omega_-}(e_{a,b}) \le A_{\Omega_+}(e_{a,b}).
\end{equation}
Then $\Omega_- \subset \Omega_+$.
\end{lem}

\begin{proof}
This is a slight modification of the proof of Lemma~\ref{lem:Aab}.
\end{proof}

\begin{proof}[Proof of Theorem~\ref{thm:concave2}.]
Suppose there exists a $2$-anchored symplectic embedding
\[
(\varphi,\Sigma_1,\Sigma_2): (X_{\Omega_-},\gamma_1,\gamma_2) \longrightarrow (X_{\Omega_+},\gamma_1,\gamma_2).
\]
We need to show that $\Omega_1\subset\Omega_2$. By Lemma~\ref{lem:Aabconcave}, it is enough to show that if $a,b>0$ are relatively prime positive integers, then the action inequality \eqref{eqn:Aabconcave} holds. This follows from Lemma~\ref{lem:phiabconcave} by the same argument as in the proof of Theorem~\ref{thm:convex2}.
\end{proof}